\documentclass[11pt]{article}

\usepackage{setspace}
\usepackage[english]{babel}
\usepackage[utf8]{inputenc}
\usepackage{amsmath}
\usepackage{scalerel}
\usepackage{amssymb}
\usepackage{graphicx}
\usepackage{float}
\usepackage{textcomp}
\usepackage{fancyhdr}
\usepackage{amsthm}
\usepackage{amssymb}
\usepackage{graphicx}
\usepackage{appendix}
\usepackage{indentfirst}
\usepackage{mathtools}
\mathtoolsset{showonlyrefs,showmanualtags}
\usepackage{bm}
\usepackage[maxbibnames=99,doi=false,isbn=false,url=false,eprint=false]{biblatex}
\AtEveryBibitem{%
  \clearfield{issn}%
  \clearfield{doi}%
  \clearfield{url}%
  \clearfield{eprint}%
  \clearfield{mrnumber}%
  \clearfield{mrclass}%
}
\renewbibmacro{in:}{}
\usepackage{csquotes}
\usepackage{enumerate}
\usepackage{xcolor}
\usepackage{subcaption} 
\usepackage{hyperref}
\usepackage{multicol}

\hypersetup{
    colorlinks=true,
    linkcolor=black,
    citecolor=black,
    pdftitle={}
}

\newcommand{\E}{\bm{\mathrm E}}
\newcommand{\Prob}{\bm{\mathrm{P}}}

\newtheorem{theorem}{Theorem}[section]
\newtheorem{corollary}[theorem]{Corollary}

\newtheorem{lemma}[theorem]{Lemma}
\newtheorem{proposition}[theorem]{Proposition}

\newtheorem{remark}[theorem]{Remark}

\usepackage[margin=1in]{geometry}
\allowdisplaybreaks
\numberwithin{equation}{section}
\usepackage{graphicx} 

\title{Tests of independence for pairs of paths of non-stationary Gaussian processes}
\author{Philip A. Ernst\footnote{
Department of Mathematics,
Imperial College London,
London SW7 2AZ, UK,
\texttt{p.ernst@imperial.ac.uk}}
\and Frederi G. Viens\footnote{
Department of Statistics,
Rice University,
Houston, TX 77005, US,
\texttt{viens@rice.edu}}
\and Shuo Yan\footnote{
Department of Mathematics,
Imperial College London,
London SW7 2AZ, UK,
\texttt{s.yan@imperial.ac.uk}}}
\date{}

\begin{document}
\maketitle
\begin{abstract}
For a pair of paths of continuous-time stochastic processes, the notion of Pearson correlation $\rho(T)$, over a time interval $[0,T]$, is an easy extension of the same notion of empirical correlation for two time series of fixed length $n$, and indeed they are explicitly tied to each other when observing the stochastic processes at a fixed time frequency (time step $=T/n$). When the processes are self-similar, such as with standard Brownian motion, or fractional Brownian motion (fBm), a simple observation leads to showing that the univariate distribution of $\rho(T)$ is constant as $T$ changes. Therefore, such highly non-stationary processes cannot be treated the same way that i.i.d. data or stationary time series are treated, where in many cases, $\rho(T)$ is known to converge to 0 as $T$ tends to infinity if the stationary time series are independent. The statistician G. Udny Yule had noticed this phenomenon empirically in 1926, giving rise to the notion of ``nonsense correlation’’ for the behavior of $\rho(T)$ which is frequently far from 0, and ``highly dispersed’’ as a distribution, irrespective of the value of $T$.

In this paper, using $T=1$ without loss of generality, we show that a notion of convergence does hold, can be analyzed in detail, and even gives rise to procedures for testing the independence of two paths of standard or fractional Brownian motions. This is the convergence to $\rho = \rho(1)$ of its discrete-observation version $\rho_n$, as the observation frequency $n$ tends to infinity. Using explicit computations based on characteristic functions of the trivariate elements that enter the expressions for the ratios defining $\rho$ and $\rho_n$, we show that $\rho_n - \rho$ converges to 0 at the rate $1/n$, and that the fluctuations of $n(\rho_n - \rho)$ are asymptotically non-normal. Based on a fully explicit representation, we show that the limit in distribution of $n(\rho_n-\rho)$, conditionally on the paths, is a normal distribution with mean and variance being non-linear functionals of the paths.

Because of this, the testing procedure which arises from our limit theorem can be understood as a procedure based on ordinary Gaussian quantiles, where the asymptotic mean is non-zero and is explicitly computable from the data, and where the asymptotic covariance is also explicitly computable from the data. Perhaps most notably from an applicability perspective, the rate of scaling in the conditionally normal convergence is $n^{-1}$ rather than the usual $n^{-1/2}$, opening a path to significantly greater power in testing procedures compared to those based on ordinary central limit theorems.\\
\indent \textbf{Keywords:} Inference in infinite dimensions,
correlation of paths,
conditionally normal fluctuations,
Yule’s nonsense correlation,
stochastic processes,
second Wiener chaos,
self-similarity,
random walk\\
\indent \textbf{Mathematics Subject Classification:} 60B10,
60H05,
62F03,
62H20 
    
\end{abstract}
\section{Introduction}

For a pair of paths of real-valued stochastic processes $X_1(t)$ and $X_2(t)$ over the time interval $[0,T]$, the empirical Pearson correlation is defined as
\begin{equation}
    \rho(T)=\frac{\int_0^TX_1(t)X_2(t)dt-\int_0^TX_1(t)dt\int_0^TX_2(t)dt}{\sqrt{\int_0^TX_1(t)^2dt-\left(\int_0^TX_1(t)dt\right)^2}\cdot\sqrt{\int_0^TX_2(t)^2dt-\left(\int_0^TX_2(t)dt\right)^2}}.
\end{equation}
This definition is motivated by this expression being the limit of its discrete (Riemann sum) version, which coincides with the classical Pearson correlation for the paths $X_1,X_2$ observed in discrete time. See expression \eqref{rhon} further below. In practice, this metric is often used to determine whether there is linear dependence between two time series, which is suitable for many stationary processes.
See \cite{MR418379,MR845894} and reference therein for treatment for univariate time series, \cite{MR1463322,MR2118942} for multivariate ARMA time series, and \cite{DPhEV,DEV,MR4815990} for a study with stochastic processes in continuous time.
However, it was observed by Yule \cite{5c7b6f25-ed11-3745-8118-935d66a8f3d3} in 1926 that the distribution of $\rho(T)$ is highly dispersed for certain types of non-stationary stochastic processes such as standard Brownian motions (Wiener processes), even for large $T$, and assessing (in)dependence based solely on $\rho(T)$ is methodologically unsound.
Still, this observation did not draw much attention from domain specialists until 1986, when the surprisingly frequent occurrence of ``spurious regression'' in applied econometric literature was discovered and confirmed (\cite{10.1214/ss/1177009870,GRANGER1974111,https://doi.org/10.1111/j.1468-0084.1986.mp48003001.x,MR867979,}).
Nevertheless, even recently, such erroneous tests based on the empirical Pearson correlation have still been widely adopted. 

This methodological error potentially leads to serious consequences: a large empirical correlation is detected for independent highly non stationary time series, and an association, or even an attribution, is declared, the independence notwithstanding. A possible example of this phenomenon can be inferred from the reference to paleoclimatology in \cite{mcshane}. That particular example points to how dangerous these methodological errors can be, when they call into question the conclusions by scientists studying problems of global importance. This serious problem is indeed one major motivation for us to study $\rho(T)$.\\ 
\indent In the current work, we provide theoretical results for testing (in)dependence between pairs of paths of highly non-stationary stochastic processes, in the most fundamental case - two jointly Gaussian processes, using the empirical Pearson correlation coefficient $\rho(T)$. A ``common-sense" idea is, instead of using the processes themselves, to compute the empirical correlation using the increments of the processes, which works in the case of Brownian motion, or more generally in the case of L\'evy processes, since the increments are stationary and independent.
However, this idea is limited to particular situations and does not hold in general. 
For processes with long memory such as fractional Brownian motion (fBm), passing to its increment process (fractional Gaussian
noise, or fGn), which is still highly correlated in time, may present no mathematical advantage over a path-based test, even if the data are sampled from fBm and fGn perfectly.
Moreover, using increments would defeat the purpose of trying to exploit the fact that data is presented to us as a time series or process path. 
Single-path data, a common occurrence in longitudinal studies, is rarely produced from an independent sampling procedure. 

This issue with fGn illustrates a broader point. The use of an entire path for testing, rather than its increments process, bases decisions on the path's macroscopic properties, rather than its short-time-scale regularity and correlation behavior, or its asymptotic long memory behavior. The latter high-frequency and asymptotic properties are highly prone to model mis-specification, making any increments-based testing procedure unworkable. The risk of model mis-specification based on the global properties of an entire non-stationary path, such as the path's self-similarity, which holds for fBm, is presumably far lower.    

A related area to our work is the study of distance correlation, which, along with the classical Pearson correlation, is a special case of a general notion of covariance with respect to stochastic processes introduced in \cite{MR2752127}.
It serves as a natural alternative of Pearson correlation between two real-valued random variables, as it allows for arbitrary dimensions and measures all forms of dependence rather than just linear dependence. The concept was first introduced by \cite{MR4205649} for univariate random variables, and by \cite{MR2298886} and \cite{MR2382665} for multivariate random variables.
Since then, distance correlation has been further explored in a series of papers. See, for instance, \cite{MR4319239,MR2752127, MR2979766, MR3269983, MR3053543}. 

The notion of distance correlation can be generalized to metric space \cite{MR3127883} (cf. \cite{MR3127866}), allowing, in particular, the consideration of stochastic processes, which has also been extensively studied recently. 
For details, see \cite{MR4140528,MR2752129,MR3745391} and reference therein.
For example, the authors of \cite{MR4140528} considered distance correlation of discretized stochastically continuous and bounded stochastic processes on the unit interval.

Another research direction in distance correlation focuses on stationary time series. For the case of univariate time series, we refer the reader to \cite{MR3635048,MR2915095}; for the case of multivariate time series, we refer the reader to \cite{MR3804406}.
In the above works, auto-distance correlation functions (ADCF) are used to test serial dependence.
Furthermore, in \cite{betken2024testindependencelongrangedependent,MR3779711}, the authors used auto-distance correlation functions and cross-distance correlation functions (CDCF) to test independence between two stationary time series.
Note that, in contrast with the distance correlation defined on the metric space of processes, as described in the preceding paragraph, the ADCF and the CDCF are defined in the usual Euclidean space of the processes observed at points.
Another popular option for testing independence between two time series is to use a Reproducing Kernel Hilbert Space based statistic (\cite{NIPS2007_d5cfead9,MR3127866,MR4270386}).

The advantage of distance correlation over Pearson correlation is that it can detect more than linear dependence. 
However, the distance correlation defined between two stochastic processes is not applicable to testing if given only two observed paths, and the ADCF and the CDCF presume the stationarity of the processes.
The focus of the current work is to measure dependence between two jointly Gaussian processes, in which case, linear independence is equivalent to independence. 
It is therefore reasonable to remain within the scope of Pearson correlation for simplicity.
Indeed, it turns out that $\rho(T)$ converges to the true correlation as the time horizon $T$ tends to infinity, in the case of stationary processes with good mixing properties, for example, AR$(1)$ processes, as shown in \cite{MR3526245} (see Theorem 8.3.1).
However, it is noted in \cite{MR3670196} that $\rho(T)$ has a constant non-trivial distribution rather than converge to  zero as $T\to\infty$ for independent standard Brownian motions (see \cite{MR4815990} for details on this distribution and \cite{MR4594215} for a parallel result in the discrete case; also see \cite{DEV} where the same constant-distribution property is noted for fBm and for all self-similar processes). 
Therefore, the Pearson correlation $\rho(T)$ itself is not sufficient for a hypothesis test for independence in this case.
To tackle this scenario, unlike the results in the papers we mentioned above, we will propose a new methodology based on the discrete version of empirical Pearson correlation to test for (in)dependence between common non-stationary Gaussian processes. Our methodology is based on remarkable new mathematical properties for the fluctuations of the difference between $\rho(T)$ and its discretized version $\rho_n(T)$ defined below in \eqref{rhon};  we establish these results with $T=1$ without loss of generality.

More specifically, we are interested in the time-discretization asymptotics for empirical Pearson correlation of two most commonly studied non-stationary Gaussian processes - standard Brownian motion and fractional Brownian motion (fBm). In addition to non-stationarity, they are both self-similar processes with continuous paths and stationary increments. Thus, as mentioned, and as noted in \cite{DEV}, $\rho(T)$ has a constant distribution regardless of the time horizon $T$, justifying the use of $T=1$, and confirming that there is no asymptotic distribution theory that can be developed directly from $\rho(T)$ as $T$ gets large. This is all consistent with the original observation by G. Udny Yule, that $\rho(T)$ is heavily dispersed even for large $T$, and we now know that this is mathematically verified in both the standard and fBm cases.

Despite this perfect parallel in both the standard and fBm cases, with increments which are stationary in both cases, a key difference between standard Brownian motion and fBm is that, the increments are independent in the former case, while they are highly correlated in the fBm case. In fact, fBm with Hurst parameter $H>1/2$ is commonly used as an archetype of a time series exhibiting long-range dependence. This typically creates many technical difficulties in the analysis of fBm compared to standard Brownian motion. This is because of the severe lack of the Markov property for fBm, which cannot be resolved using any multivariate considerations, and is also because there are sometimes insurmountable technicalities in trying to immerse fBm into a martingale framework. However in this paper, we show that our mathematical analysis techniques can handle both cases in rather explicit fashion, though not simultaneously, as a separate analysis is needed for fBm. 
 
As we said, we are now allowed to focus on the study of $\rho(1)$ to understand general situations.
As a discrete approximation of $\rho:=\rho(1)$, we define $\rho_n$ as
\begin{equation}
\label{rhon}
    \frac{\frac{1}{n}\sum_{k=0}^{n-1}X_1(k/n)X_2(k/n)-\frac{1}{n}\sum_{k=0}^{n-1}X_1(k/n)\cdot \frac{1}{n}\sum_{k=0}^{n-1}X_2(k/n)}{\sqrt{\frac{1}{n}\sum_{k=0}^{n-1}X_1(k/n)^2-\left(\frac{1}{n}\sum_{k=0}^{n-1}X_1(k/n)\right)^2}\cdot\sqrt{\frac{1}{n}\sum_{k=0}^{n-1}X_2(k/n)^2-\left(\frac{1}{n}\sum_{k=0}^{n-1}X_2(k/n)\right)^2}}.
\end{equation} As mentioned, this $\rho_n$ converges to $\rho$ as $n\to\infty$. While not the purpose of this paper, one can prove rather easily, using a coupling which was introduced in \cite{MR4594215}, that the convergence holds almost surely and in $L^1(\Omega)$.  
Remarkably, it turns out that the convergence can be used to reveal (in)dependence relations between two observed processes, by delving explicitly into the  nature and quantification of the fluctuations of $\rho_n-\rho$, which we establish in this article, and which can be interpreted as non-Gaussian, and as conditionally Gaussian.
Even more remarkably, unlike any standard central limit theorem, we will show that the rate of this convergence is of order $1/n$, thus that the fluctuations of interest are those of $n(\rho_n-\rho)$. That $1/n$ is the correct scaling to find the fluctuations of $\rho_n-\rho$ was already known from the moment calculations in Section 5 of \cite{MR4594215}. Herein, we show that $n(\rho_n-\rho)$ has a limit in law which can be described explicitly in two equivalent ways; a non-Gaussian law in the direct product of the first and second Wiener chaos associated with a product probability space between an exogenous Wiener space and the Gaussian space of the original paths $X_1,X_2$, where the part involving the exogenous Wiener process requires modulation by a non-linear term coming from a delta method, where the exogenous process does not appear; or as a conditionally Gaussian law, conditional on the paths $X_1,X_2$, where the Gaussian character comes only from the aforementioned exogenous Wiener process, and all other terms, including the modulating non-linearity, are functionals of the original paths $X_1,X_2$ only. Finally and perhaps most importantly in practice, we show that the asymptotics of $n(\rho_n-\rho)$ make for an implementable hypothesis test in practice, where the mean and covariance structure of the conditional Gaussian limit, on which the test is based, depend on the actual paths $X_1,X_2$, not merely on the law of that pair of paths. With an attention to potential practical implementation, we also show that our method can be applied to design a hypothesis test based only on discrete-time observations of $X_1,X_2$, for example, by observing $n(\rho_n-\rho_{2n})$.

The structure of the article is as follows. In the remainder of this section, we introduce notation, formulate the main results, and discuss hypothesis testing in practice. In Section \ref{sec:sbm}, we prove the results for standard Brownian motion. In Section \ref{sec:fbm}, we prove the results for fractional Brownian motion with Hurst parameter $H>1/2$. In Section \ref{sec:numerics}, we present numerical simulations.
\subsection{Notation and main results}
Suppose that we have two stochastic processes $X_1$ and $X_2$.
Define $F(a,b,c)=a/\sqrt{bc}$, and functionals on the space of $\mathbf C([0,1])$ and $\mathbf C([0,1])^2$:
\begin{equation}
\label{def:functionals}
    \begin{aligned}
        D(f)&=\int_0^1 f(t)^2dt-\left(\int_0^1f(t)dt\right)^2,\\
        D_n(f)&=\frac{1}{n}\sum_{k=0}^{n-1}f(k/n)^2-\left(\frac{1}{n}\sum_{k=0}^{n-1} f(k/n)\right)^2,\\
        A(f,g)&=\int_0^1 f(t)g(t)dt-\left(\int_0^1f(t)dt\right)\left(\int_0^1g(t)dt\right),\\
        A_n(f,g)&=\frac{1}{n}\sum_{k=1}^nf(k/n)g(k/n)-\frac{1}{n^2}\sum_{k=1}^nf(k/n)\sum_{k=1}^ng(k/n).
    \end{aligned}
\end{equation}
Then, the empirical variance, covariance, and correlation coefficient of $X_1$ and $X_2$ can be represented in the following way.
\begin{equation}
\label{def:empirical}
\begin{aligned}
    \text{Variance - continuous: }\mathcal X^{i,i}=D(X_i),\quad &\text{Variance - discrete: }\mathcal X_{n}^{i,i}=D_n(X_i),\quad i\in\{1, 2\},\\
    \text{Covariance - continuous: }\mathcal X^{1,2}=A(X_1,X_2),\quad &\text{Covariance - discrete: }\mathcal X^{1,2}_n=A_n(X_1,X_2),\\
    \text{Correlation - continuous: }\rho=F(\mathcal X^{1,2},\mathcal X^{1,1},\mathcal X^{2,2}),\quad &\text{Correlation - discrete: }\rho_n=F(\mathcal X_n^{1,2},\mathcal X_n^{1,1},\mathcal X_n^{2,2}).
\end{aligned}
\end{equation}
Define the following functions, with $r\in[-1,1]$ and $H\in(1/2,1)$:
\begin{equation}
\label{def:musigma}
    \begin{aligned}
        \bm\mu(f,g)&=\left(\frac{1}{2}f(1)\bar g+\frac{1}{2}g(1)\bar f-\frac{1}{2}f(1)g(1),f(1)\bar f-\frac{1}{2}f(1)^2,g(1)\bar g-\frac{1}{2}g(1)^2\right)^{\mathrm{T}},\\
        \mu(f,g)&=\nabla F(A(f,g),D(f),D(g))\cdot\bm\mu(f,g),\\
        \sigma^r(f,g)&=\frac{1}{\sqrt{12}}\cdot\frac{\sqrt{(D(f)D(g)-A(f,g)^2)(D(f)+D(g)-2rA(f,g))}}{D(f)D(g)},\\
        \sigma^{r,H}(f,g)&=\sigma_H\cdot\frac{\sqrt{(D(f)D(g)-A(f,g)^2)(D(f)+D(g)-2rA(f,g))}}{D(f)D(g)},
    \end{aligned}
\end{equation}where $\bar f=\int_0^1 f(t)dt$, $\bar g=\int_0^1g(t)dt$, and $\sigma_H$ is a constant that depends on $H$ only. For each $H>1/2$, $\sigma_H$ can be computed as an explicit integral (see Lemma~\ref{lem:variance} and Equation \eqref{def:sigmaH}).
We are ready to formulate the present paper's two key results as Theorem \ref{sthm:cov} and Theorem \ref{fthm:cov} below.
\begin{theorem}
    \label{sthm:cov}
    Let $\mathcal W_1$ and $\mathcal W_2$ be two standard Brownian motions that are jointly Gaussian with constant correlation coefficient $r$.
    Let $\mathcal X^{i,j}$ and $\mathcal X^{i,j}_n$ for $i,j\in\{1,2\}$ and corresponding $\rho$ and $\rho_n$ be defined as in \eqref{def:empirical} with $X_i=\mathcal W_i$ for $i\in\{1,2\}$.
    Let $\mathcal U=(\mathcal X^{1,2},\mathcal X^{1,1},\mathcal X^{2,2})^{\mathrm{T}}$ and $\mathcal U_n=(\mathcal X^{1,2}_n,\mathcal X^{1,1}_n,\mathcal X^{2,2}_n)^{\mathrm{T}}$. Then, the following results hold.
    \begin{enumerate}
        \item There exists a decomposition for difference of empirical covariance 
    $$n(\mathcal U_n-\mathcal U)=\bm\mu(\mathcal W_1,\mathcal W_2)+\mathcal Z_n,$$
    and $(\mathcal Z_n,\mathcal W_1,\mathcal W_2)\to (\mathcal Z,\mathcal W_1,\mathcal W_2)$ in distribution, where, conditional on $(\mathcal W_1,\mathcal W_2)$, $\mathcal Z$ is normally distributed with zero mean.
    \item There exists a decomposition for difference of empirical correlation
    \begin{equation}
        n(\rho_n-\rho)=\mu(\mathcal W_1,\mathcal W_2)+\mathcal Y_n,
    \end{equation}and $(\mathcal Y_n,\mathcal W_1,\mathcal W_2)\to (\mathcal Y,\mathcal W_1,\mathcal W_2)$ in distribution, where, conditional on $(\mathcal W_1,\mathcal W_2)$, $\mathcal Y$ is normally distributed with zero mean.
    \item Moreover, as $n\to\infty$,
    \begin{equation}
        \frac{n(\rho_n-\rho)-\mu(\mathcal W_1,\mathcal W_2)}{\sigma^r(\mathcal W_1,\mathcal W_2)}\to\mathcal N(0,1)\quad\text{in distribution.}
    \end{equation}
    \end{enumerate}
\end{theorem}

\begin{theorem}
    \label{fthm:cov}
    Let $\mathcal B^H_1$ and $\mathcal B^H_2$ be two fractional Brownian motions with Hurst parameter $H>1/2$, which are jointly Gaussian with constant correlation coefficient $r$.
    Let $\mathcal X^{i,j}$ and $\mathcal X^{i,j}_n$ for $i,j\in\{1,2\}$ and corresponding $\rho$ and $\rho_n$ be defined as in \eqref{def:empirical} with $X_i=\mathcal B^H_i$ for $i\in\{1,2\}$.
    Let $\mathcal U=(\mathcal X^{1,2},\mathcal X^{1,1},\mathcal X^{2,2})^{\mathrm{T}}$ and $\mathcal U_n=(\mathcal X^{1,2}_n,\mathcal X^{1,1}_n,\mathcal X^{2,2}_n)^{\mathrm{T}}$. Then, the following results hold.
    \begin{enumerate}
        \item There exists a decomposition for difference of empirical covariance 
    $$n(\mathcal U_n-\mathcal U)=\bm\mu(\mathcal B^H_1,\mathcal B^H_2)+n^{\frac{1}{2}-H}\mathcal Z_n^H,$$
    and $(\mathcal Z_n^H,\mathcal B^H_1,\mathcal B^H_2)\to (\mathcal Z^H,\mathcal B^H_1,\mathcal B^H_2)$ in distribution, where, conditional on $(\mathcal B^H_1,\mathcal B^H_2)$, $\mathcal Z^H$ is normally distributed with zero mean.
    \item There exists a decomposition for difference of empirical correlation
    \begin{equation}
        n(\rho_n-\rho)=\mu(\mathcal B^H_1,\mathcal B^H_2)+n^{\frac{1}{2}-H}\mathcal Y_n^H,
    \end{equation}and $(\mathcal Y_n^H,\mathcal B^H_1,\mathcal B^H_2)\to (\mathcal Y^H,\mathcal B^H_1,\mathcal B^H_2)$ in distribution, where, conditional on $(\mathcal B^H_1,\mathcal B^H_2)$, $\mathcal Y^H$ is normally distributed with zero mean.
    \item Moreover, as $n\to\infty$,
    \begin{equation}
        \frac{n^{H+\frac{1}{2}}(\rho_n-\rho)-n^{H-\frac{1}{2}}\mu(\mathcal B^H_1,\mathcal B^H_2)}{\sigma^{r,H}(\mathcal B^H_1,\mathcal B^H_2)}\to\mathcal N(0,1)\quad\text{in distribution.}
    \end{equation}
    \end{enumerate}
\end{theorem}

\begin{remark}
    The explicit form of $\sigma_H$, the covariance matrix of $\mathcal Z$ and $\mathcal Z^H$, and the variance of $\mathcal Y$ and $\mathcal Y^H$ are given in Lemma~\ref{lem:variance}, with a simplified formula \eqref{def:sigmaH}, and Propositions~\ref{sprop:conj6},  \ref{fprop:conj6}, \ref{sprop:conj7}, and \ref{fprop:conj7}, respectively.
\end{remark}

\subsection{Discrete version of the results}
The following results, which do not require continuous-time observation of the processes, can be helpful when conducting hypothesis test in practice. The proofs are analogous to those of Theorems~\ref{sthm:cov} and \ref{fthm:cov}.
Here we only present the results without giving the proofs.
Define the following functions that only depend on $f$ and $g$ on the set $\{k/n: 0\leq k\leq n\}$:
\begin{align}
    \bm\mu_n(f,g)&=\begin{pmatrix}
        \frac{1}{4n}f(1)\sum_{k=0}^{n-1}g(k/n)+\frac{1}{4n}g(1)\sum_{k=0}^{n-1}f(k/n)-\frac{1}{4}f(1)g(1)\\
        \frac{1}{2n}f(1)\sum_{k=0}^{n-1}f(k/n)-\frac{1}{4}f(1)^2\\
        \frac{1}{2n}g(1)\sum_{k=0}^{n-1}g(k/n)-\frac{1}{4}g(1)^2
    \end{pmatrix},\\
    \mu_n(f,g)&=\nabla F(A_n(f,g),D_n(f),D_n(g))\cdot\bm\mu_n(f,g),\\
    \sigma_n^r(f,g)&=\frac{1}{4}\cdot\frac{\sqrt{(D_n(f)D_n(g)-A_n(f,g)^2)(D_n(f)+D_n(g)-2rA_n(f,g))}}{D_n(f)D_n(g)},\\
    \sigma_n^{r,H}(f,g)&=\sigma^d_H\cdot\frac{\sqrt{(D_n(f)D_n(g)-A_n(f,g)^2)(D_n(f)+D_n(g)-2rA_n(f,g))}}{D_n(f)D_n(g)},
\end{align}where $\sigma_H^d$ is a constant that depends on $H$ only. 
The explicit integral formula for $\sigma_H^d$ (here $\{\cdot\}$ denotes the fractional part) is
\begin{gather}
\label{def:sigma_Hd}
        \sigma_H^d=\sqrt{\frac{H(2H-1)}{4}\int_0^{1/2}\left[R_H^d\left(\gamma+\frac{1}{2}\right)-R_H^d(\gamma)\right]d\gamma},\text{ where}\\
        R_H^d(\gamma)=\frac{\gamma^{2H}}{H(2H-1)}-\frac{\gamma^{2H-1}}{2H-1}+\frac{\gamma^{2H-2}}{6}-(2-2H)\int_0^\infty\left(\{x\}^2-\{x\}+\frac{1}{6}\right)(\gamma+x)^{2H-3}dx.
    \end{gather}
\begin{proposition}
\label{sprop:discrete}
     Let $\mathcal W_1$ and $\mathcal W_2$ are two standard Brownian motions that are jointly Gaussian with constant correlation coefficient $r$.
    Let $\rho_n$ be defined as in \eqref{def:empirical} with $X_i=\mathcal W_i$ for $i\in\{1,2\}$.
    Then
    \begin{equation}
        \frac{n(\rho_n-\rho_{2n})-\mu_n(\mathcal W_1,\mathcal W_2)}{\sigma_n^r(\mathcal W_1,\mathcal W_2)}\to\mathcal N(0,1)\quad\text{in distribution.}
    \end{equation}
\end{proposition}

\begin{proposition}
\label{fprop:discrete}
     Let $\mathcal B^H_1$ and $\mathcal B^H_2$ are two fractional Brownian motions, with Hurst parameter $H>1/2$, which are jointly Gaussian with constant correlation coefficient $r$.
    Let $\rho_n$ be defined as in \eqref{def:empirical} with $X_i=\mathcal B^H_i$ for $i\in\{1,2\}$.
    Then
    \begin{equation}
        \frac{n^{H+\frac{1}{2}}(\rho_n-\rho_{2n})-n^{H-\frac{1}{2}}\mu_n(\mathcal B^H_1,\mathcal B^H_2)}{\sigma_n^{r,H}(\mathcal B^H_1,\mathcal B^H_2)}\to\mathcal N(0,1)\quad\text{in distribution.}
    \end{equation}
\end{proposition}

\subsection{Hypothesis testing}
Assuming that one has access to the continuous-time observation of the processes, then Part 3 of Theorem~\ref{sthm:cov} and Theorem~\ref{fthm:cov} can be applied to conduct a hypothesis test of independence for standard Brownian motions and fractional Brownian motions, respectively.
For example, for standard Brownian motions, a hypothesis test based on Part 3 of Theorem~\ref{sthm:cov} can be carried out as follows.
\begin{enumerate}[(1)]
    \item Observe paths $\mathcal W_1$ and $\mathcal W_2$;
    \item Based on $\mathcal W_1$ and $\mathcal W_2$, produce $\rho$, $\mu(\mathcal W_1,\mathcal W_2)$, and $\rho_n$ for large $n$.
    \item Under $H_0$, produce $\sigma^r(\mathcal W_1,\mathcal W_2)$ based on $\mathcal W_1$ and $\mathcal W_2$;
    \item Reject $H_0$ if $|n(\rho_n-\rho)-\mu(\mathcal W_1,\mathcal W_2)|/{\sigma^r(\mathcal W_1,\mathcal W_2)}$ is too large.
\end{enumerate}
However, in practice, continuous-time observation is often not possible.
In this more realistic scenario, one can either (a) use a discrete-time observation of a much higher frequency $m\gg n$ to produce (approximation of) $\rho$, $\mu(\mathcal W_1,\mathcal W_2)$, and $\sigma^r(\mathcal W_1,\mathcal W_2)$ under $H_0$, produce $\rho_n$ using a discrete-time observation of a relatively lower frequency $n$, and conduct a test using those approximated quantities; or (b) conduct a hypothesis test based on Proposition~\ref{sprop:discrete} for standard Brownian motions or Proposition~\ref{fprop:discrete} for fractional Brownian motions, in a similar fashion to that described above.

\begin{remark}
    In practice, since $\sigma_H$ and $\sigma_H^d$ are rather complicated to compute,  one can instead first sample (e.g. using Monte Carlo) two fractional Brownian motions and calculate the value of $\sigma_H$ and $\sigma_H^d$, which are constants that depend only on the Hurst parameter $H$, based on the convergence results above. 
\end{remark}

\section{Results on standard Brownian motions}
This section will prove Theorem~\ref{sthm:cov}. Suppose that we have two standard Brownian motions $\mathcal W_1$ and $\mathcal W_2$ which are jointly Gaussian with constant correlation coefficient $r$. We will first write the empirical variance and covariance in the form of stochastic integrals, prove weak convergence of the integrands and hence the integrals, and apply the delta method to derive the asymptotic normality concerning empirical correlation coefficient, as stated in Part 3 of Theorem~\ref{sthm:cov}.
We begin with only one Brownian motion.
\label{sec:sbm}
    \subsection{Representation of the empirical variance as iterated It\^o integrals}
Suppose that $\mathcal  W$ is a Brownian motion.
It can be verified that 
\begin{equation}
\label{eq:s1}
    \begin{aligned}
        n(D_n(\mathcal W)-D(\mathcal W))=2n\int_0^1\int_0^t(M_n(s,t)-M(s,t))d\mathcal W(s)d\mathcal W(t)+\frac{1}{6n},
    \end{aligned}
\end{equation}
where 
\begin{equation}
\begin{aligned}
    M_n(s,t)&=\frac{\lceil ns\rceil}{n}\wedge\frac{\lceil nt\rceil}{n}-\frac{\lceil ns\rceil}{n}\cdot\frac{\lceil nt\rceil}{n},\\
    M(s,t)&=s\wedge t-st.
\end{aligned}
\end{equation}
Then the r.h.s. of \eqref{eq:s1} can be written as the sum of three terms:
\begin{equation}
    \label{eq:sdecom}
\begin{aligned}
    nD_n^1&=2\int_0^1\int_0^t(\lceil ns\rceil-ns)(1-t)d\mathcal W(s)d\mathcal W(t),\\
    nD_n^2&=-2\int_0^1\int_0^t(\lceil nt\rceil-nt)sd\mathcal W(s)d\mathcal W(t),\\
    nD_n^3&=-2\int_0^1\int_0^t\frac{1}{n}(\lceil nt\rceil-nt)(\lceil ns\rceil-ns)d\mathcal W(s)d\mathcal W(t)+\frac{1}{6n}.
\end{aligned}
\end{equation}
The last term $nD_n^3$ converges to $0$ in $L^2$ as $n\to\infty$. The first two terms can be written as Wiener integrals with proper kernels
\begin{align}
    nD_n^1&=I_2\left((1-s\vee t)(\lceil n(s\wedge t)\rceil - n(s\wedge t))\right),\\
    nD_n^2&=-I_2\left( (s\wedge t)(\lceil n(s\vee t)\rceil - n(s\vee t))\right)\\
        &=-\tilde I_2\left((1-s\vee t)(n(s\wedge t)-\lfloor n(s\wedge t)\rfloor)\right)
    \end{align}
where $\tilde I_2$ is the Wiener integral w.r.t. the Brownian motion $\tilde {\mathcal W}_t:=\mathcal W_{1-t}-\mathcal W_1$.
They can be written as iterated It\^o integrals with similar forms but with different Brownian motions
\begin{equation}
    \begin{aligned}
     nD_n^1&=2\int_0^1(1-t)\int_0^t(\lceil ns\rceil-ns)d\mathcal W(s)d\mathcal W(t),\\
        nD_n^2&=-2\int_0^1(1-t)\int_0^t(ns-\lfloor ns\rfloor)d\tilde{\mathcal  W}(s)d\tilde {\mathcal W}(t).
    \end{aligned}
\end{equation}
To obtain the limiting distribution of the sum, we would need the results of the joint distributions of random variables involved in the expressions of $nD_n^1$ and $nD_n^2$.

Define
\begin{equation}
    \mathcal P_n(t)=\int_0^t\left(\lceil ns\rceil-ns-\frac{1}{2}\right)d{\mathcal W}(s),
\end{equation}
\begin{equation}
    \mathcal Q_n(t)=\int_0^t\left(ns-\lfloor ns\rfloor-\frac{1}{2}\right)d\tilde {\mathcal W}(s)=-\int_{1-t}^1\left(\lceil ns\rceil-ns-\frac{1}{2}\right)d\mathcal W(s).
\end{equation}
Then
\begin{equation}
\label{eq:s2}
\begin{aligned}
    nD_n^1 &= \int_0^1 (1-t){\mathcal W}(t)d{\mathcal W}(t)+2\int_0^1(1-t)\mathcal P_n(t)d{\mathcal W}(t)\\
    &=\frac{1}{2}I_2(1-s\vee t)+2\int_0^1(1-t)\mathcal P_n(t)d{\mathcal W}(t).\\
    nD_n^2&=-\int_0^1(1-t)\tilde {\mathcal W}(t)d\tilde {\mathcal W}(t)-2\int_0^1(1-t)\mathcal Q_n(t)d\tilde {\mathcal W}(t)\\
    &=-\frac{1}{2}\bar I_2(1-s\vee t)-2\int_0^1(1-t)\mathcal Q_n(t)d\tilde {\mathcal W}(t)\\
    &=-\frac{1}{2} I_2(s\wedge t)-2\int_0^1(1-t)\mathcal Q_n(t)d\tilde {\mathcal W}(t).
\end{aligned}
\end{equation}
    \subsection{Weak convergence of the integrands}
    \begin{lemma}
\label{lem:weak_conv1}
    $(\mathcal P_n(t),{\mathcal W}(t),\mathcal Q_n(t),\tilde {\mathcal W}(t))$ converges to $(\frac{1}{\sqrt{12}}\mathcal{B}(t),{\mathcal W}(t),\frac{1}{\sqrt{12}}\tilde {\mathcal{B}}(t),\tilde {\mathcal W}(t))$ in $\mathbf{C}([0,1],\mathbb R^4)$ in distribution, as $n\to\infty$, where $\mathcal B$ and $\mathcal W$ are independent Brownian motions, $\tilde {\mathcal W}(t)={\mathcal W}({1-t})-{\mathcal W}(1)$, $\tilde {\mathcal{B}}(t)={\mathcal{B}}(1-t)-{\mathcal{B}}(1)$. 
\end{lemma}
\begin{proof}
    The tightness is straightforward to verify, so it remains to prove the finite-dimensional-distribution convergence. 
    Since the process $(P_n(t),{\mathcal W}_t,Q_n(t),\tilde {\mathcal W}_t)$ is jointly Gaussian, the weak convergence is equivalent to the convergence of the covariance matrix. 
    As $n\to\infty$, 
    \begin{enumerate}[(1)]
        \item \begin{equation}
            \E \mathcal P_n(t)\mathcal P_n(s)=\int_0^{s\wedge t}\left(\lceil nu\rceil-nu-\frac{1}{2}\right)^2du\to\frac{1}{12}(s\wedge t)=\frac{1}{12}\E \mathcal B(t)\mathcal B(s);
        \end{equation}
        \item \begin{equation}
            \E \mathcal P_n(t)\mathcal W(s)=\int_0^{s\wedge t}\left(\lceil nu\rceil-nu-\frac{1}{2}\right)du\to0=\frac{1}{\sqrt{12}}\E\mathcal B(t)\mathcal W(s);
        \end{equation}
        \item \begin{equation}
            \begin{aligned}
                \E \mathcal P_n(t)\mathcal Q_n(s)&=-\int_{t\wedge(1-s)}^t\left(\lceil nu\rceil-nu-\frac{1}{2}\right)^2du\to\frac{1}{12}((1-s-t)\wedge0)=\frac{1}{12}\E\mathcal B(t)\tilde{\mathcal B}(s);
            \end{aligned}
        \end{equation}
        \item 
        \begin{equation}
            \begin{aligned}
                \E \mathcal P_n(t)\tilde{\mathcal W}(s)&=-\int_{t\wedge(1-s)}^t\left(\lceil nu\rceil-nu-\frac{1}{2}\right)du\to0=\frac{1}{\sqrt{12}}\E\mathcal B(t)\tilde{\mathcal W}(s).
            \end{aligned}
        \end{equation}
    \end{enumerate}
    The convergence of $\E\mathcal W(t)\mathcal Q_n(s)$, $\E\tilde{\mathcal W}(t)\mathcal Q_n(s)$, and $\E\mathcal Q_n(t)\mathcal Q_n(s)$ can be proved in the same way, and the other required results of convergence are trivial.
\end{proof}
    \subsection{Weak convergence of the covariance and correlation: Proof of Theorem~\ref{sthm:cov}}
\subsubsection{Covariance: independence case}
\label{ssec:indep}
    Now suppose that we have two independent Brownian motion $W_1$ and $W_2$. Then the same steps lead to the following lemma.
\begin{lemma}
Define
\begin{align}
    P_n(t)&=\int_0^t\left(\lceil ns\rceil-ns-\frac{1}{2}\right)dW_1(s),\\
    Q_n(t)&=\int_0^t\left(ns-\lfloor ns\rfloor-\frac{1}{2}\right)d\tilde W_1(s)=-\int_{1-t}^1\left(\lceil ns\rceil-ns-\frac{1}{2}\right)dW_1(s),\\
    R_n(t)&=\int_0^t\left(\lceil ns\rceil-ns-\frac{1}{2}\right)dW_2(s),\\
    S_n(t)&=\int_0^t\left(ns-\lfloor ns\rfloor-\frac{1}{2}\right)d\tilde W_2(s)=-\int_{1-t}^1\left(\lceil ns\rceil-ns-\frac{1}{2}\right)dW_2(s).
\end{align}
The process
\begin{equation}
    (P_n(t),Q_n(t),R_n(t),S_n(t),W_1(t),\tilde W_1(t),W_2(t),\tilde W_2(t))
\end{equation}
converges in distribution to 
\begin{equation}
    \left(\frac{1}{\sqrt{12}}{\mathcal{B}}_1(t),\frac{1}{\sqrt{12}}\tilde {\mathcal{B}}_1(t),\frac{1}{\sqrt{12}}{\mathcal{B}}_2(t),\frac{1}{\sqrt{12}}\tilde {\mathcal{B}}_2(t),W_1(t),\tilde W_1(t),W_2(t),\tilde W_2(t)\right),
\end{equation}
where ${\mathcal{B}}_1,{\mathcal{B}}_2,W_1,W_2$ are independent Brownian motion and $\tilde X_i(t)=X_i(1-t)-X_i(1)$, for $X\in\{\mathfrak B,W\}$ and $i\in\{1,2\}$.
\end{lemma}
Recall the functionals defined in \eqref{def:functionals}. For simplicity of notation, we define
\begin{gather}
\label{notation:covariance}
    X^{i,i}=D(W_i),\quad X_{n}^{i,i}=D_n(W_i),\quad\text{for }i=1, 2,\\
    X^{1,2}=A(W_1,W_2),\quad X^{1,2}_n=A_n(W_1,W_2).
\end{gather}
As in \eqref{eq:s1} and \eqref{eq:sdecom}, one can verify that, for $i\in\{1,2\}$,
\begin{equation}
\begin{aligned}
    n(X^{i,i}_n-X^{i,i})&=2n\int_0^1\int_0^t(M_n(s,t)-M(s,t))dW_i(s)dW_i(t)+\frac{1}{6n}.
\end{aligned}
\end{equation}
One can also obtain similar expression for $n(X^{1,2}_n-X^{1,2})$:
\begin{equation}
\label{eq:sxn12}
\begin{aligned}
        n(X^{1,2}_n-X^{1,2})&=n\int_0^1\int_0^t(M_n(s,t)-M(s,t))dW_1(s)dW_2(t)\\
        &\quad+n\int_0^1\int_0^t(M_n(s,t)-M(s,t))dW_2(s)dW_1(t).
\end{aligned}
\end{equation}
\begin{lemma}
\label{lem:sconj6_indep}
    Let $U=(X^{1,2},X^{1,1},X^{2,2})^{\mathrm T}$ and $U_n=(X^{1,2}_n,X^{1,1}_n,X^{2,2}_n)^{\mathrm T}$. If $W_1$ and $W_2$ are independent, then $n(U_n-U)=V+Z_n$ and $(Z_n,W_1,W_2)\to (Z,W_1,W_2)$ in distribution, where
    \begin{equation}
    \begin{aligned}
        V
        &=\begin{pmatrix}
            \frac{1}{2}W_1(1)\bar W_2+\frac{1}{2}W_2(1)\bar W_1-\frac{1}{2}W_1(1)W_2(1)\\
            W_1(1)\bar W_1-\frac{1}{2}W_1(1)^2\\
            W_2(1)\bar W_2-\frac{1}{2}W_2(1)^2
        \end{pmatrix},
    \end{aligned}
    \end{equation}
    \begin{equation}
        Z=\begin{pmatrix}
            \frac{1}{\sqrt{12}}\int_0^1(\bar W_1-W_1(t))d{\mathcal{B}}_2(t)+\frac{1}{\sqrt{12}}\int_0^1(\bar W_2-W_2(t))d{\mathcal{B}}_1(t)\\
            \frac{1}{\sqrt{3}}\int_0^1(\bar W_1-W_1(t))d{\mathcal{B}}_1(t)\\
            \frac{1}{\sqrt{3}}\int_0^1(\bar W_2-W_2(t))d{\mathcal{B}}_2(t)
        \end{pmatrix},
    \end{equation}where $\bar W_i = \int_0^1W_i(t)dt$, $i\in\{1,2\}$, and $\mathcal B_i$, $i\in\{1,2\}$, are independent Brownian motions.
Moreover, conditional on $(W_1,W_2)$, $Z\sim\mathcal N(0,\Sigma)$, where
\begin{equation}
    \Sigma=\frac{1}{12}\begin{pmatrix}
        X^{1,1}+X^{2,2}&2X^{1,2}&2X^{1,2}\\
        2X^{1,2}&4X^{1,1}&0\\
        2X^{1,2}&0&4X^{2,2}
    \end{pmatrix}.
\end{equation}
\end{lemma}
\begin{remark}
    Note that $W_1$ and $W_2$ are Wiener paths we observe, and $n(U_n-U)$ can be computed explicitly based on $W_1$ and $W_2$ for every fixed $n$; $\mathcal B_1$ and $\mathcal B_2$ are standard Wiener processes independent of $W_1$ and $W_2$ that only appear in the limit in distribution as $n\to\infty$.
\end{remark}
\begin{proof}
    As shown in \eqref{eq:sdecom} and \eqref{eq:s2}, we can decompose $n(X^{1,1}_n-X^{1,1})$ and $n(X^{2,2}_n-X^{2,2})$ in the following way
    \begin{equation}
    \begin{aligned}
        n(X^{1,1}_n-X^{1,1})&=\int_0^1\int_0^t(1-s-t)dW_1(s)dW_1(t)+2\int_0^1(1-t)P_n(t)dW_1(t)\\
        &\quad-2\int_0^1(1-t)Q_n(t)d\tilde W_1(t)+E_n^{1,1},\\
        n(X^{2,2}_n-X^{2,2})&=\int_0^1\int_0^t(1-s-t)dW_2(s)dW_2(t)+2\int_0^1(1-t)R_n(t)dW_2(t)\\
        &\quad-2\int_0^1(1-t)S_n(t)d\tilde W_2(t)+E_n^{2,2},
    \end{aligned}
    \end{equation}where $E^{i,i}_n=-2\int_0^1\int_0^t\frac{1}{n}(\lceil nt\rceil-nt)(\lceil ns\rceil-ns)dW_i(s)dW_i(t)+\frac{1}{6n}\to0$ in $L^2$, $i\in\{1,2\}$. 
    On the other hand, $n(X^{1,2}_n-X^{1,2})$ has a different expression. From \eqref{eq:sxn12}, we obtain
        \begin{align*}
            &n(X^{1,2}_n-X^{1,2})\\
            &=\int_0^1(1-t)\int_0^t(\lceil ns\rceil-ns)dW_1(s)dW_2(t)+\int_0^1(1-t)\int_0^t(\lceil ns\rceil-ns)dW_2(s)dW_1(t)\\
            &\quad-\int_0^1\int_0^t(\lceil nt\rceil-nt)sdW_1(s)dW_2(t)-\int_0^1\int_0^t(\lceil nt\rceil-nt)sdW_2(s)dW_1(t)\\
            &\quad-\int_0^1\int_0^t\frac{1}{n}(\lceil nt\rceil-nt)(\lceil ns\rceil-ns)dW_1(s)dW_2(t)-\int_0^1\int_0^t\frac{1}{n}(\lceil nt\rceil-nt)(\lceil ns\rceil-ns)dW_2(s)dW_1(t)\\
            &=\frac{1}{2}\int_0^1(1-t)W_1(t)dW_2(t)+\int_0^1(1-t)P_n(t)dW_2(t)\\
            &\quad+\frac{1}{2}\int_0^1(1-t)W_2(t)dW_1(t)+\int_0^1(1-t)R_n(t)dW_1(t)\\
            &\quad-\int_0^1\int_s^1(\lceil nt\rceil-nt)dW_2(t)sdW_1(s)-\int_0^1\int_s^1(\lceil nt\rceil-nt)dW_1(t)sdW_2(s)+E^{1,2}_n\\
            &=\frac{1}{2}\int_0^1(1-t)W_1(t)dW_2(t)+\int_0^1(1-t)P_n(t)dW_2(t)\\
            &\quad+\frac{1}{2}\int_0^1(1-t)W_2(t)dW_1(t)+\int_0^1(1-t)R_n(t)dW_1(t)\\
            &\quad+\frac{1}{2}\int_0^1 t(W_2(t)-W_2(1))dW_1(t)+\frac{1}{2}\int_0^1 t(W_1(t)-W_1(1))dW_2(t)\\
            &\quad+\int_0^1S_n(1-s)sdW_1(s)+\int_0^1Q_n(1-s)sdW_2(s)+E^{1,2}_n\\
            &=\frac{1}{2}\int_0^1(W_1(t)-tW_1(1))dW_2(t)+\frac{1}{2}\int_0^1(W_2(t)-tW_2(1))dW_1(t)\\
            &\quad+\int_0^1((1-t)P_n(t)+tQ_n(1-t))dW_2(t)+\int_0^1((1-t)R_n(t)+tS_n(1-t))dW_1(t)+E^{1,2}_n,
        \end{align*}where
        \begin{equation}
        \begin{aligned}
            E^{1,2}_n&=-\int_0^1\int_0^t\frac{1}{n}(\lceil nt\rceil-nt)(\lceil ns\rceil-ns)dW_1(s)dW_2(t)\\
            &\quad-\int_0^1\int_0^t\frac{1}{n}(\lceil nt\rceil-nt)(\lceil ns\rceil-ns)dW_2(s)dW_1(t)
        \end{aligned}
        \end{equation}converges to zero in $L^2$.
    Now let us prove that $V$ has the desired form.
    For $i\in\{1,2\}$, define $Y_t^i = \frac{1}{2}(1-2t)W_i(t)^2$. By It\^o's formula,
    \begin{equation}
        \begin{aligned}
            dY_t^i&=(1-2t)W_i(t)dW_i(t)+\frac{1}{2}(1-2t-2W_i(t)^2)dt.
        \end{aligned}
    \end{equation}
    Thus, 
    \begin{equation}
        \begin{aligned}
            \int_0^1(1-2t)W_i(t)dW_i(t)&=Y^i_1-Y_0^i-\frac{1}{2}\int_0^1(1-2t-2W_i(t)^2)dt\\
            &=-\frac{1}{2}W_i(1)^2+\int_0^1W_i(t)^2dt.
        \end{aligned}
    \end{equation}
    By integration by parts, 
        \begin{align}
            &\int_0^1\int_0^t(1-s-t)dW_i(s)dW_i(t)\\
            &=\int_0^1\left[(1-t)W_i(t)-\int_0^t sdW_i(s)\right]dW_i(t)\\
            &=\int_0^1\left[(1-2t)W_i(t)+\int_0^tW_i(s)ds\right]dW_i(t)\\
            &=-\frac{1}{2}W_i(1)^2+\int_0^1W_i(t)^2dt+W_i(1)\int_0^1W_i(t)dt-\int_0^1W_i(t)^2dt\\
            &=-\frac{1}{2}W_i(1)^2+W_i(1)\int_0^1W_i(t)dt.
        \end{align}
    That proves the second and the third components of $V$. 
    For the first component, we can see that
    \begin{equation}
        \begin{aligned}
            &\frac{1}{2}\int_0^1(W_1(t)-tW_1(1))dW_2(t)+\frac{1}{2}\int_0^1(W_2(t)-tW_2(1))dW_1(t)\\
            &=\frac{1}{2}\left(\int_0^1W_1(t)dW_2(t)+\int_0^1W_2(t)dW_1(t)\right)-\frac{1}{2}W_1(1)\int_0^1tdW_2(t)-\frac{1}{2}W_2(1)\int_0^1tdW_1(t)\\
            &=\frac{1}{2}W_1(1)W_2(1)-\frac{1}{2}W_1(1)W_2(1)+\frac{1}{2}W_1(1)\int_0^1W_2(t)dt-\frac{1}{2}W_1(1)W_2(1)+\frac{1}{2}W_2(1)\int_0^1W_1(t)dt\\
            &=\frac{1}{2}W_1(1)\int_0^1W_2(t)dt+\frac{1}{2}W_2(1)\int_0^1W_1(t)dt-\frac{1}{2}W_1(1)W_2(1).
        \end{aligned}
    \end{equation}
    So far, we proved that $n(U_n-U)=V+Z_n$, where 
    \begin{equation}
        Z_n=\begin{pmatrix}
            \int_0^1((1-t)P_n(t)+tQ_n(1-t))dW_2(t)+\int_0^1((1-t)R_n(t)+tS_n(1-t))dW_1(t)\\
            2\int_0^1(1-t)P_n(t)dW_1(t)-2\int_0^1(1-t)Q_n(t)d\tilde W_1(t)\\
            2\int_0^1(1-t)R_n(t)dW_2(t)-2\int_0^1(1-t)S_n(t)d\tilde W_2(t)
        \end{pmatrix}+
        \begin{pmatrix}
            E^{1,2}_n\\
            E^{1,1}_n\\
            E^{2,2}_n
        \end{pmatrix},
    \end{equation}where to last term converges to $0$ in $L^2$.
    By Theorem  7.10 in \cite{Kurtz1996}, $(Z_n,W_1,W_2)$ converges weakly to $(Z,W_1,W_2)$ where 
    \begin{align*}
        Z&=\begin{pmatrix}
            \frac{1}{\sqrt{12}}\int_0^1 ((1-t){\mathcal{B}}_1(t)+t\tilde {\mathcal{B}}_1(1-t))dW_2(t)+
            \frac{1}{\sqrt{12}}\int_0^1 ((1-t){\mathcal{B}}_2(t)+t\tilde {\mathcal{B}}_2(1-t))dW_1(t)\\
            \frac{1}{\sqrt{3}}\int_0^1(1-t){\mathcal{B}}_1(t)dW_1(t)-\frac{1}{\sqrt{3}}\int_0^1(1-t)\tilde {\mathcal{B}}_1(t)d\tilde W_1(t)\\
            \frac{1}{\sqrt{3}}\int_0^1(1-t){\mathcal{B}}_2(t)dW_2(t)-\frac{1}{\sqrt{3}}\int_0^1(1-t)\tilde {\mathcal{B}}_2(t)d\tilde W_2(t)
        \end{pmatrix}\\
        &=\begin{pmatrix}
            \frac{1}{\sqrt{12}}\int_0^1({\mathcal{B}}_1(t)-t{\mathcal{B}}_1(1))dW_2(t)+\frac{1}{\sqrt{12}}\int_0^1({\mathcal{B}}_2(t)-t{\mathcal{B}}_2(1))dW_1(t)\\
            \frac{1}{\sqrt{3}}\int_0^1({\mathcal{B}}_1(t)-t{\mathcal{B}}_1(1))dW_1(t)\\
            \frac{1}{\sqrt{3}}\int_0^1({\mathcal{B}}_2(t)-t{\mathcal{B}}_2(1))dW_2(t)
        \end{pmatrix}\\
        &=\begin{pmatrix}
            \frac{1}{\sqrt{12}}\int_0^1(\int_0^1W_1(s)ds-W_1(t))d{\mathcal{B}}_2(t)+\frac{1}{\sqrt{12}}\int_0^1(\int_0^1W_2(s)ds-W_2(t))d{\mathcal{B}}_1(t)\\
            \frac{1}{\sqrt{3}}\int_0^1(\int_0^1W_1(s)ds-W_1(t))d{\mathcal{B}}_1(t)\\
            \frac{1}{\sqrt{3}}\int_0^1(\int_0^1W_2(s)ds-W_2(t))d{\mathcal{B}}_2(t)
        \end{pmatrix},
    \end{align*}
    where the last equality is due to integration by parts: for $i, j\in\{1,2\}$,
    \begin{equation}
        \begin{aligned}
            &\int_0^1({\mathcal{B}}_i(t)-t{\mathcal{B}}_i(1))dW_j(t)\\
            &=\int_0^1{\mathcal{B}}_i(t)dW_j(t)-\mathcal{B}_i(1)\int_0^1 tdW_j(t)\\
            &=\mathcal B_i(1)W_j(1)-\int_0^1W_j(t)d\mathcal B_i(t)-\mathcal B_i(1)W_j(1)+\mathcal B_i(1)\int_0^1W_j(t)dt\\
            &=\int_0^1\left(\int_0^1W_j(s)ds-W_j(t)\right)d{\mathcal{B}}_i(t).
        \end{aligned}
    \end{equation}
    Then, the desired covariance matrix conditional on $W_1$ and $W_2$ follows from the It\^o isometry.
    \end{proof}
\subsubsection{Covariance: general case}
\label{ssec:dep}
    Now we are ready to consider the general case formulated in Theorem~\ref{sthm:cov}.
    Since we have two standard Brownian motions $\mathcal W_1$ and $\mathcal W_2$ which are jointly Gaussian with constant correlation coefficient $r$, $\mathcal W_1$ and $\frac{\mathcal W_2-r\mathcal W_1}{\sqrt{1-r^2}}$ are two independent Brownian motions, and Lemma~\ref{lem:sconj6_indep} can be applied.
Recall that the empirical covariance can be represented as
\begin{gather}
    \mathcal X^{i,i}=D(\mathcal W_i),\quad \mathcal X_{n}^{i,i}=D_n(\mathcal W_i),\quad\text{for }i=1, 2,\\
    \mathcal X^{1,2}=A(\mathcal W_1,\mathcal W_2),\quad \mathcal X^{1,2}_n=A_n(\mathcal W_1,\mathcal W_2).
\end{gather}
Let $W_1 = \mathcal W_1$ and $W_2=\frac{\mathcal W_1-r\mathcal W_2}{\sqrt{1-r^2}}$. Recalling the notation in \eqref{notation:covariance}, then, by linearity of integration and summation,
\begin{equation}
\label{eq:slinearity}
    \begin{pmatrix}
        \mathcal X_n^{1,2}-\mathcal X^{1,2}\\
        \mathcal X_n^{1,1}-\mathcal X^{1,1}\\
        \mathcal X_n^{2,2}-\mathcal X^{2,2}
    \end{pmatrix}=\begin{pmatrix}
        \sqrt{1-r^2}&r&0\\
        0&1&0\\
        2r\sqrt{1-r^2}&r^2&1-r^2
    \end{pmatrix}\cdot
    \begin{pmatrix}
        X_n^{1,2}-X^{1,2}\\
        X_n^{1,1}-X^{1,1}\\
        X_n^{2,2}-X^{2,2}
    \end{pmatrix}=:A_r(U_n-U).
\end{equation}
\begin{proposition}
    \label{sprop:conj6}
        Let $\mathcal U=(\mathcal X^{1,2},\mathcal X^{1,1},\mathcal X^{2,2})$ and $\mathcal U_n=(\mathcal X^{1,2}_n,\mathcal X^{1,1}_n,\mathcal X^{2,2}_n)$. If $\mathcal W_1$ and $\mathcal W_2$ are jointly Gaussian with constant coefficient $r$, then $n(\mathcal U_n-\mathcal U)=\mathcal V+\mathcal Z_n$ and $(\mathcal Z_n,\mathcal W_1,\mathcal W_2)\to (\mathcal Z,\mathcal W_1,\mathcal W_2)$ in distribution, where
    \begin{gather}
        \mathcal V=\begin{pmatrix}
            \frac{1}{2}\mathcal W_1(1)\bar{\mathcal W}_2+\frac{1}{2}\mathcal W_2(1)\bar{\mathcal W}_1-\frac{1}{2}\mathcal W_1(1)\mathcal W_2(1)\\
            \mathcal W_1(1)\bar{\mathcal W}_1-\frac{1}{2}\mathcal W_1(1)^2\\
            \mathcal W_2(1)\bar{\mathcal W}_2-\frac{1}{2}\mathcal W_2(1)^2
        \end{pmatrix},\\
        \mathcal Z=\begin{pmatrix}
            \frac{\sqrt{1-r^2}}{\sqrt{12}}\int_0^1(\bar {\mathcal W}_1-{\mathcal W}_1(t))d{\mathcal{B}}_2(t)+\frac{r}{\sqrt{12}}\int_0^1(\bar{\mathcal W}_1-{\mathcal W}_1(t))d\mathcal B_1(t)+\frac{1}{\sqrt{12}}\int_0^1(\bar W_2-W_2(t))d{\mathcal{B}}_1(t)\\
            \frac{1}{\sqrt{3}}\int_0^1(\bar{\mathcal W}_1-{\mathcal W}_1(t))d{\mathcal{B}}_1(t)\\
            \frac{\sqrt{1-r^2}}{\sqrt{3}}\int_0^1(\bar {\mathcal W}_2-{\mathcal W}_2(t))d{\mathcal{B}}_2(t)+\frac{r}{\sqrt{3}}\int_0^1(\bar {\mathcal W}_2-{\mathcal W}_2(t))d{\mathcal{B}}_1(t)
        \end{pmatrix},
    \end{gather}where $\bar{\mathcal W}_i=\int_0^1\mathcal W_i(t)dt$,
    and, conditional on $(\mathcal W_1,\mathcal W_2)$, $\mathcal Z\sim \mathcal N(0,\Sigma_r)$, where
    \begin{equation}
        \Sigma_r=\frac{1}{12}\begin{pmatrix}
            \mathcal X^{1,1}+\mathcal X^{2,2}+2r \mathcal X^{1,2}&2\mathcal X^{1,2}+2r \mathcal X^{1,1}& 2\mathcal X^{1,2}+2r \mathcal X^{2,2}\\
            2\mathcal X^{1,2}+2r \mathcal X^{1,1}&4\mathcal X^{1,1}&4r \mathcal X^{1,2}\\
            2\mathcal X^{1,2}+2r \mathcal X^{2,2}&4r \mathcal X^{1,2}& 4\mathcal X^{2,2}
        \end{pmatrix}.
    \end{equation}
    \end{proposition}
\begin{proof}
    By \eqref{eq:slinearity} and Lemma~\ref{lem:sconj6_indep}, we have $n(\mathcal U_n-\mathcal U)=nA_r(U_n-U)=A_r(V+Z_n)$.
    Again, by linearity of integration and summation, $\mathcal V=A_r V$. Define $\mathcal Z_n=A_rZ_n$ and $\mathcal Z=A_rZ$. 
    It follows that $n(\mathcal U_n-\mathcal U)=\mathcal V+\mathcal Z_n$ and $(\mathcal Z_n,\mathcal W_1,\mathcal W_2)\to (\mathcal Z,\mathcal W_1,\mathcal W_2)$ in distribution.
    We compute $\mathcal Z$ explicitly:
    \begin{align}
            \mathcal Z_1&=\frac{\sqrt{1-r^2}}{\sqrt{12}}\int_0^1\left(\int_0^1W_1(s)ds-W_1(t)\right)d{\mathcal{B}}_2(t)+\frac{\sqrt{1-r^2}}{\sqrt{12}}\int_0^1\left(\int_0^1W_2(s)ds-W_2(t)\right)d{\mathcal{B}}_1(t)\\
            &\quad+\frac{r}{\sqrt{3}}\int_0^1\left(\int_0^1W_1(s)ds-W_1(t)\right)d{\mathcal{B}}_1(t)\\
            &=\frac{\sqrt{1-r^2}}{\sqrt{12}}\int_0^1\left(\int_0^1\mathcal W_1(s)ds-\mathcal W_1(t)\right)d{\mathcal{B}}_2(t)+\frac{r}{\sqrt{12}}\int_0^1\left(\int_0^1\mathcal W_1(s)ds-\mathcal W_1(t)\right)d{\mathcal{B}}_1(t)\\
            &\quad+\frac{1}{\sqrt{12}}\int_0^1\left(\int_0^1\mathcal W_2(s)ds-\mathcal W_2(t)\right)d{\mathcal{B}}_1(t),\\
            \mathcal Z_2&=\frac{1}{\sqrt{3}}\int_0^1\left(\int_0^1\mathcal W_1(s)ds-\mathcal W_1(t)\right)d{\mathcal{B}}_1(t),\\
            \mathcal Z_3&=\frac{r\sqrt{1-r^2}}{\sqrt{3}}\int_0^1\left(\int_0^1W_1(s)ds-W_1(t)\right)d{\mathcal{B}}_2(t)+\frac{r\sqrt{1-r^2}}{\sqrt{3}}\int_0^1\left(\int_0^1W_2(s)ds-W_2(t)\right)d{\mathcal{B}}_1(t)\\
            &\quad+\frac{r^2}{\sqrt{3}}\int_0^1\left(\int_0^1 W_1(s)ds-W_1(t)\right)d{\mathcal{B}}_1(t)+\frac{1-r^2}{\sqrt{3}}\int_0^1\left(\int_0^1W_2(s)ds-W_2(t)\right)d{\mathcal{B}}_2(t)\\
            &=\frac{\sqrt{1-r^2}}{\sqrt{3}}\int_0^1\left(\int_0^1\mathcal W_2(s)ds-\mathcal W_2(t)\right)d{\mathcal{B}}_2(t)+\frac{r}{\sqrt{3}}\int_0^1\left(\int_0^1\mathcal W_2(s)ds-\mathcal W_2(t)\right)d{\mathcal{B}}_1(t).
    \end{align}
    The desired covariance matrix $\Sigma_r$ conditional on $\mathcal W_1$ and $\mathcal W_2$ follows from the It\^o isometry.
\end{proof}

\subsubsection{Empirical correlation}
\label{ssec:corr}
Let $\mathcal U$, $\mathcal U_n$, $\mathcal V$, and $\mathcal Z$ be defined as in Proposition~\ref{sprop:conj6}. 
Recall that $F(a,b,c)=a/\sqrt{bc}$, $\rho=F(\mathcal U)$, and $\rho_n=F(\mathcal U_n)$.
    \begin{proposition}
    \label{sprop:conj7}
    As $n\to\infty$, $(n(\rho_n-\rho)-\nabla F(\mathcal U)\cdot \mathcal V,\mathcal W_1,\mathcal W_2)\to(\nabla F(\mathcal U)\cdot \mathcal Z,\mathcal W_1,\mathcal W_2)$ in distribution, and, conditional on $(\mathcal W_1,\mathcal W_2)$, $$\nabla F(\mathcal U)\cdot \mathcal Z\sim\mathcal N\left(0,\frac{(\mathcal X^{1,1}\mathcal X^{2,2}-(\mathcal X^{1,2})^2)(\mathcal X^{1,1}+\mathcal X^{2,2}-2r\mathcal X^{1,2})}{12(\mathcal X^{1,1}\mathcal X^{2,2})^2}\right).$$
\end{proposition}
\begin{proof}
    By Taylor's expansion, we obtain that
    \begin{equation}
        \begin{aligned}
            n(\rho_n-\rho)&=n(F(\mathcal U_n)-F(\mathcal U))\\
            &=\nabla F(\mathcal U)\cdot n(\mathcal U_n-\mathcal U)+[\nabla F(\mathcal U)-\nabla F(g(\mathcal U_n,\mathcal U))]\cdot n(\mathcal U_n-\mathcal U),
        \end{aligned}
    \end{equation}where $g(x,y)$ is a point on the line segment connecting $x$ and $y$.
    It suffices to prove the second term on the right-hand side converges to $0$ in probability and verify that 
    \begin{equation}
    \label{eq:verifyf}
    \nabla F(\mathcal U)^*\Sigma_r\nabla F(\mathcal U)=\frac{(\mathcal X^{1,1}\mathcal X^{2,2}-(\mathcal X^{1,2})^2)(\mathcal X^{1,1}+\mathcal X^{2,2}-2r\mathcal X^{1,2})}{(\mathcal X^{1,1}\mathcal X^{2,2})^2}.
    \end{equation}
    The verification is a simple matrix multiplication calculation, and so the details are omitted.
    
    Fix $\kappa>0$ and $\varepsilon$ arbitrarily small. It follows from the previous calculations that $n(U_n-U)$ has uniformly bounded $L^2$ norm. Hence, we can find $M>0$ sufficiently large such that for all $n$
    \begin{equation}
    \label{eq:e1}
        \Prob(n|\mathcal U_n-\mathcal U|>M)<\varepsilon/3.
    \end{equation}
    Note that $\nabla F$ is continuous outside the set $S:=\{(x,y,z):xyz=0\}$, thus uniformly continuous in any compact set that does not intersect with $S$ and, moreover, $\Prob(\mathcal U\in S)=0$. Let 
    \[L=\{x\in\mathbb R^3:y\in\mathbb R^3,|x-y|<\delta\Rightarrow|\nabla F(x)-\nabla F(y)|<\kappa/M|\}.\]
    Then, there exists $\delta>0$ such that \begin{equation}
    \label{eq:e2}
    \Prob(\mathcal U\in L)>1-\varepsilon/3.
    \end{equation} 
    Again, since $n(\mathcal U_n-\mathcal U)$ has uniformly bounded $L^2$ norm, there exists $N>0$ such that for all $n>N$
    \begin{equation}
    \label{eq:e3}
        \Prob(|\mathcal U_n-\mathcal U|>\delta)<\varepsilon/3.
    \end{equation}
    Combining \eqref{eq:e1}, \eqref{eq:e2}, and \eqref{eq:e3}, with $M,\delta,N$ chosen as above, we have for $n>N$
    \begin{align}
        &\Prob(|[\nabla F(\mathcal U)-\nabla F(g(\mathcal U_n,\mathcal U))]\cdot n(\mathcal U_n-\mathcal U)|\leq\kappa)\\
        &\quad\geq \Prob(n|\mathcal U_n-\mathcal U|\leq M, |\nabla F(\mathcal U)-\nabla F(g(\mathcal U_n,\mathcal U))|\leq\kappa/M)\\
        &\quad\geq \Prob(|\nabla F(\mathcal U)-\nabla F(g(\mathcal U_n,\mathcal U))|\leq\kappa/M)-\varepsilon/3\\
        &\quad\geq \Prob(\mathcal U\in L,|\mathcal U-\mathcal U_n|\leq\delta)-\varepsilon/3\\
        &\quad>1-\varepsilon/3-\varepsilon/3-\varepsilon/3\\
        &\quad =1-\varepsilon.
    \end{align}
    The result now follows from Proposition~\ref{sprop:conj6}.
\end{proof} 

\begin{corollary}
\label{scor:unconditional}
    Let $\mathcal U$ and $\mathcal V$ be defined as in Proposition~\ref{sprop:conj6}.
    Recalling the definition in \eqref{def:musigma}, we have \begin{gather}
    {\sigma^r(\mathcal W_1,\mathcal W_2)}^2 = \frac{(\mathcal X^{1,1}\mathcal X^{2,2}-(\mathcal X^{1,2})^2)(\mathcal X^{1,1}+\mathcal X^{2,2}-2r\mathcal X^{1,2})}{12(\mathcal X^{1,1}\mathcal X^{2,2})^2}.
    \end{gather}
    Then \[
    \frac{n(\rho_n-\rho)-\nabla F(\mathcal U)\cdot \mathcal V}{\sigma^r(\mathcal W_1,\mathcal W_2)}\to\mathcal N(0,1)
    \] in distribution.
\end{corollary}
\begin{proof}
    Let $\xi_n = n(\rho_n-\rho)-\nabla F(\mathcal U)\cdot\mathcal V$ and $\xi = \nabla F(\mathcal U)\cdot \mathcal Z$.
    Since $(\xi_n,\mathcal W_1, \mathcal W_2)\to(\xi,\mathcal W_1, \mathcal W_2)$, we have
    \begin{align}
        \E \exp\left(it\cdot\frac{\xi_n}{\sigma^r(\mathcal W_1,\mathcal W_2)}\right)&\to\E\exp\left(it\cdot\frac{\xi}{\sigma^r(\mathcal W_1,\mathcal W_2)}\right)\\
        &=\E\left[\E\left(\exp\left(i\cdot\frac{t}{\sigma^r(\mathcal W_1,\mathcal W_2)}\cdot\xi\right)\bigg|\mathcal W_1,\mathcal W_2\right)\right]\\
        &=\E\exp\left(-\frac{\sigma^r(\mathcal W_1,\mathcal W_2)^2\cdot(t/\sigma^r(\mathcal W_1,\mathcal W_2))^2}{2}\right)\\
        &=\exp(-t^2/2).\qedhere
    \end{align}
\end{proof}

Theorem~\ref{sthm:cov} follows immediately from Propositions~\ref{sprop:conj6} and \ref{sprop:conj7} and Corollary~\ref{scor:unconditional}.

\section{Results on fractional Brownian motions}
\label{sec:fbm}
In this section, we will prove the results for the case of fractional Brownian motions with Hurst parameter $H>1/2$.
Note that we will re-define some letters from Section \ref{sec:sbm} to denote analogous terms concerning fractional Brownian motions instead of standard Brownian motions.
Suppose that $\mathcal  B^H(t)$ is a fractional Brownian motion with the Hurst parameter $H>1/2$.
Let $\mathcal E$ be the set of step functions on $[0,1]$.
Consider the Hilbert space $\mathcal H$ to be defined as the closure of $\mathcal E$ with respect to the scalar product
\begin{equation}
    \langle \bm 1_{[0,t]},\bm 1_{[0,s]}\rangle_{\mathcal H}=\frac{1}{2}(t^{2H}+s^{2H}-|t-s|^{2H}).
\end{equation}
For any functions $\varphi,\psi\in\mathcal H$, we have that
\begin{equation}
    \langle \varphi,\psi\rangle_{\mathcal H}=H(2H-1)\int_0^1\int_0^1|r-u|^{2H-2}\varphi_r\psi_udrdu.
\end{equation}
The mapping $B:\bm 1_{[0,t]}\mapsto \mathcal B^H(t)$ can be extended to an isometry between $\mathcal H$ and the Gaussian space associated with $\mathcal B^H$.
Then $\{B(\varphi),\varphi\in\mathcal H\}$ is an isonormal Gaussian process.
The steps of the proof are similar to those in Section~\ref{sec:sbm}, except that an additional multiplier $n^{H-\frac{1}{2}}$ is required to reveal the nontrivial normality in the limit.

\subsection{Representation of the covariance as iterated Stratonovich integrals}
\label{fsec:representation}
As in the case of standard Brownian motion, we will write $D(\mathcal B^H)$ and $D_n(\mathcal B^H)$ as stochastic integrals with respect to $\mathcal B^H(t)$.
Since $H>1/2$, we can use the path-wise Riemann-Stieltjes integral (Young integral), which coincides with the Stratonovich integral, denoted by $\int X_t\circ d\mathcal B^H(t)$, and with divergence integral if the integrand is deterministic.
More precisely, for the continuous integrals in $D(\mathcal B^H)$, we have
\begin{align}
    \int_0^1 \mathcal B^H(t)dt&=\lim_{n\to\infty}\frac{1}{n}\sum_{k=1}^n\mathcal B^H(k/n)\\
    &=\lim_{n\to\infty}\frac{1}{n}\sum_{k=0}^{n-1}(n-k)(\mathcal B^H((k+1)/n)-\mathcal B^H({k/n}))\\
    &=\int_0^1 (1-t)\circ d\mathcal B^H(t).
\end{align}
By the chain rule of the Stratonovich integral,
\begin{equation}
    \left(\int_0^1 \mathcal B^H(t)dt\right)^2=2\int_0^1(1-t)\int_0^t(1-s)\circ d\mathcal B^H(s)\circ d\mathcal B^H(t).
\end{equation}
On the other hand,
\begin{align}
    \int_0^1 (\mathcal B^H(t))^2dt&=\lim_{n\to\infty}\frac{1}{n}\sum_{k=1}^n\mathcal B^H(k/n)^2\\
    &=\lim_{n\to\infty}\frac{1}{n}\sum_{k=0}^{n-1}(n-k)(\mathcal B^H({(k+1)/n}))-\mathcal B^H({k/n}))(\mathcal B^H({(k+1)/n}))+\mathcal B^H({k/n}))\\
    &=2\int_0^1 (1-t)\mathcal B^H(t)\circ d\mathcal B^H(t).
\end{align}
For the discrete sums in $D_n(\mathcal B^H)$, we have
\begin{align}
    \frac{1}{n}\sum_{k=0}^{n-1}\mathcal B^H({k/n})&=\frac{1}{n}\sum_{k=0}^{n-1}(n-k-1)(\mathcal B^H({(k+1)/n})-\mathcal B^H(k/n))\\
    &=\int_0^1\left(1-\frac{\lceil nt\rceil}{n}\right)\circ d\mathcal B^H(t).
\end{align}
By the chain rule of the Stratonovich integral,
\begin{equation}
    \left(\frac{1}{n}\sum_{k=1}^n\mathcal B^H(k/n)\right)^2=2\int_0^1\left(1-\frac{\lceil nt\rceil}{n}\right)\int_0^t\left(1-\frac{\lceil ns\rceil}{n}\right)\circ d\mathcal B^H(s)\circ d\mathcal B^H(t).
\end{equation}
On the other hand, 
\begin{align}
    \frac{1}{n}\sum_{k=0}^{n-1}\mathcal B^H(k/n)^2&=\frac{1}{n}\sum_{k=0}^{n-1}(n-k-1)(\mathcal B^H((k+1)/n)^2-\mathcal B^H(k/n)^2)\\
    &=2\cdot\frac{1}{n}\sum_{k=0}^{n-1}(n-k-1)\int_{k/n}^{(k+1)/n}\mathcal B^H(t)\circ d\mathcal B^H(t)\\
    &=2\sum_{k=0}^{n-1}\int_{k/n}^{(k+1)/n}\left(1-\frac{\lceil nt\rceil}{n}\right)\mathcal B^H(t)\circ d\mathcal B^H(t)\\
    &=2\int_0^1\left(1-\frac{\lceil nt\rceil}{n}\right)\mathcal B^H(t)\circ d\mathcal B^H(t).
\end{align}
Thus, it can be verified that 
\begin{equation}
\label{feq:1}
    \begin{aligned}
        n(D_n(\mathcal B^H)-D(\mathcal B))=2n\int_0^1\int_0^t(M_n(s,t)-M(s,t))\circ d\mathcal B^H(s)\circ d\mathcal B^H(t),
    \end{aligned}
\end{equation}
where 
\begin{equation}
\begin{aligned}
    M_n(s,t)&=\frac{\lceil ns\rceil}{n}\wedge\frac{\lceil nt\rceil}{n}-\frac{\lceil ns\rceil}{n}\cdot\frac{\lceil nt\rceil}{n},\\
    M(s,t)&=s\wedge t-st.
\end{aligned}
\end{equation}
Then the r.h.s. of \eqref{feq:1} can be written as the sum of three terms
\begin{equation}
    \label{eq:decom}
\begin{aligned}
    nD_n^1&=2\int_0^1\int_0^t(\lceil ns\rceil-ns)(1-t)\circ d\mathcal B^H(s)\circ d\mathcal B^H(t),\\
    nD_n^2&=-2\int_0^1\int_0^t(\lceil nt\rceil-nt)s\circ d\mathcal B^H(s)\circ d\mathcal B^H(t),\\
    nD_n^3&=-2\int_0^1\int_0^t\frac{1}{n}(\lceil nt\rceil-nt)(\lceil ns\rceil-ns)\circ d\mathcal B^H(s)\circ d\mathcal B^H(t).
\end{aligned}
\end{equation}
The last term $nD_n^3$ converges to $0$ in $L^2$ as $n\to\infty$ with a rate of $1/n$. The first two terms can be written as iterated Stratonovich integrals with similar integrands but with different fractional Brownian motions:
\begin{equation}
    \begin{aligned}
     nD_n^1&=2\int_0^1(1-t)\int_0^t(\lceil ns\rceil-ns)\circ d\mathcal B^H(s)\circ d\mathcal B^H(t),\\
        nD_n^2&=-2\int_0^1(1-t)\int_0^t(ns-\lfloor ns\rfloor)\circ d\tilde{\mathcal B}^H(s)\circ d\tilde{\mathcal B}^H(t),
    \end{aligned}
\end{equation}where $\tilde{\mathcal B}^H(t)=\mathcal B^H(1-t)-\mathcal B^H(1)$.
To obtain the limiting distribution of the sum, we need the results of the joint distributions of random variables involved in the expressions of $nD_n^1$ and $nD_n^2$.

Define
\begin{equation}
    \mathcal P_n(t)=\int_0^t\left(\lceil ns\rceil-ns-\frac{1}{2}\right)\circ d\mathcal B^H(s),
\end{equation}
\begin{equation}
    \mathcal Q_n(t)=\int_0^t\left(ns-\lfloor ns\rfloor-\frac{1}{2}\right) \circ d\tilde{\mathcal B}^H(s)=-\int_{1-t}^1\left(\lceil ns\rceil-ns-\frac{1}{2}\right)\circ d\mathcal B^H(s).
\end{equation}
Then
\begin{equation}
\label{eq:2}
\begin{aligned}
    nD_n^1 &= \int_0^1 (1-t){\mathcal B}^H(t)\circ d{\mathcal B}^H(t)+2\int_0^1(1-t)\mathcal P_n(t)\circ d{\mathcal B}^H(t).\\
    nD_n^2&=-\int_0^1(1-t)\tilde {\mathcal B}^H(t)\circ d\tilde {\mathcal B}^H(t)-2\int_0^1(1-t)\mathcal Q_n(t)\circ d\tilde {\mathcal B}^H(t)\\
    &=\int_0^1 t(\mathcal B^H(t)-\mathcal B^H(1))\circ d{\mathcal B}^H(t)-2\int_0^1(1-t)\mathcal Q_n(t)\circ d\tilde {\mathcal B}^H(t).
\end{aligned}
\end{equation}
    \subsection{Convergence of the integrands}
    In contrast to Section~\ref{sec:sbm}, the integrands $\mathcal P_n$ and $\mathcal Q_n$ will no longer converge to nontrivial Gaussian processes. 
    In fact, due to the positive correlation of the increments of fractional Brownian motions with Hurst parameter $H>1/2$, $\mathcal P_n$ and $\mathcal Q_n$ will converge to zero processes.
    However, with an additional multiplier $n^{H-\frac{1}{2}}$, we will see that $n^{H-\frac{1}{2}}\mathcal P_n(t)$ and $n^{H-\frac{1}{2}}\mathcal Q_n(t)$ do have nontrivial limits.
    To simplify notation, let $\alpha=2H-1$. Since $\frac{1}{2}<H<1$, we have $0<\alpha<1$.
\subsubsection{\texorpdfstring{$L^2$}{L2} convergence to zero}
    \label{sec:before_norm}
Since 
\begin{equation}
    \mathcal P_n(t)=\int_0^t\left(\lceil ns\rceil-ns-\frac{1}{2}\right)\circ d\mathcal B^H(s)=\int_0^t\left(\lceil ns\rceil-ns-\frac{1}{2}\right) d\mathcal B^H(s),
\end{equation}where the right-hand side is understood as the divergence integral, 
\begin{equation}
\begin{aligned}
    \E \mathcal P_n(t)^2&=\left\langle \bm  1_{[0,t]}\left(\lceil ns\rceil-ns-\frac{1}{2}\right),\bm 1_{[0,t]}\left(\lceil ns\rceil-ns-\frac{1}{2}\right)\right\rangle_{\mathcal H}\\
    &=H(2H-1)\int_0^t\int_0^t|s-u|^{2H-2}\left(\lceil ns\rceil-ns-\frac{1}{2}\right)\left(\lceil nu\rceil-nu-\frac{1}{2}\right)duds\\
    &=2H(2H-1)\int_0^t\int_0^s(s-u)^{2H-2}\left(\lceil nu\rceil-nu-\frac{1}{2}\right)du\cdot \left(\lceil ns\rceil-ns-\frac{1}{2}\right)ds.
\end{aligned}
\end{equation}
Let us define 
\begin{equation}
\label{def:fn}
    f_n(t) = \int_0^t(t-s)^{2H-2}\left(\lceil ns\rceil -ns-\frac{1}{2}\right)ds=\int_0^t(t-s)^{\alpha-1}\left(\lceil ns\rceil -ns-\frac{1}{2}\right)ds.
\end{equation}
Then 
\begin{equation}
    \label{eq:variance0}
    \E \mathcal P_n(t)^2 = 2H(2H-1)\int_0^tf_n(s)\left(\lceil ns\rceil -ns-\frac{1}{2}\right)ds.
\end{equation}
In fact, $f_n(t)$ can be computed relatively explicitly.
Namely, for $t\not\in\{j/n:j\in\mathbb N\}$,
\begin{align}
        f_n(t)&=\sum_{k=1}^{\lfloor nt\rfloor}\int_{(k-1)/n}^{k/n}(t-s)^{\alpha-1}\left(k-\frac{1}{2}-ns\right)ds+\int_{\lfloor nt\rfloor/n}^t(t-s)^{\alpha-1}\left(\lceil nt\rceil -\frac{1}{2}-ns\right)ds\\
        &=\sum_{k=1}^{\lfloor nt\rfloor}\int_{(k-1)/n}^{k/n}(t-s)^{\alpha-1}\left(k-\frac{1}{2}-nt\right)ds+n\sum_{k=1}^{\lfloor nt\rfloor}\int_{(k-1)/n}^{k/n}(t-s)^{\alpha}ds\\
        &\quad +\int_{\lfloor nt\rfloor/n}^t(t-s)^{\alpha-1}\left(\lceil nt\rceil -\frac{1}{2}-nt\right)ds+n\int_{\lfloor nt\rfloor/n}^t(t-s)^{\alpha}ds\\
        &=\sum_{k=1}^{\lfloor nt\rfloor}\frac{k-\frac{1}{2}-nt}{\alpha}[(t-(k-1)/n)^{\alpha}-(t-k/n)^{\alpha}]\\
        &\quad+\frac{n}{\alpha+1}\sum_{k=1}^{\lfloor nt\rfloor}[(t-(k-1)/n)^{\alpha+1}-(t-k/n)^{\alpha+1}]\\
        &\quad+\frac{\lceil nt\rceil -\frac{1}{2}-nt}{\alpha}\left(t-\frac{\lfloor nt\rfloor}{n}\right)^{\alpha}+\frac{n}{\alpha+1}\left(t-\frac{\lfloor nt\rfloor}{n}\right)^{\alpha+1}\\
        &=\frac{1/2-nt}{\alpha}t^{\alpha}+\frac{1}{\alpha}\sum_{k=1}^{\lfloor nt\rfloor-1}(t-k/n)^{\alpha}-\frac{\lfloor nt\rfloor-1/2-nt}{\alpha}(t-\lfloor nt\rfloor/n)^{\alpha}\\
        &\quad+\frac{n}{\alpha+1}(t^{\alpha+1}-(t-\lfloor nt\rfloor/n)^{\alpha+1})+\frac{\lceil nt\rceil -\frac{1}{2}-nt}{\alpha}\left(t-\frac{\lfloor nt\rfloor}{n}\right)^{\alpha}+\frac{n}{\alpha+1}\left(t-\frac{\lfloor nt\rfloor}{n}\right)^{\alpha+1}\\
        &=\frac{1}{\alpha}\sum_{k=1}^{\lfloor nt\rfloor}(t-k/n)^{\alpha}+\frac{1}{2\alpha}t^{\alpha}-\frac{n}{(\alpha+1)\alpha}t^{\alpha+1}.
    \end{align}
By invoking the Euler-Maclaurin summation formula, we can get an expansion of the sum on the right-hand side of the equality above.
    Defining $s_{\gamma}(x)=(\gamma+x)^\alpha$, $0<\gamma<1$, then
    \begin{equation}
    \label{euler1}
        \begin{aligned}
            \sum_{k=0}^{\lfloor nt\rfloor-1}s_{\gamma}(k)&=\int_0^{\lfloor nt\rfloor}s_\gamma(x) dx-\frac{1}{2}s_\gamma(x)\bigg|^{\lfloor nt\rfloor}_0+\frac{1}{12}s'_{\gamma}(x)\bigg|^{\lfloor nt\rfloor}_0+R_2(\gamma,\lfloor nt\rfloor)\\
            &=\frac{1}{1+\alpha}\left[(\gamma+\lfloor nt\rfloor)^{\alpha+1}-\gamma^{\alpha+1}\right]-\frac{1}{2}\left[(\gamma+\lfloor nt\rfloor)^{\alpha}-\gamma^{\alpha}\right]\\
            &\quad+\frac{\alpha}{12}\left[(\gamma+\lfloor nt\rfloor)^{\alpha-1}-\gamma^{\alpha-1}\right]+R_2(\lfloor nt\rfloor,\gamma),
        \end{aligned}
    \end{equation}
where $R_2(\lfloor nt\rfloor,\gamma)=-\frac{1}{2}\int_0^{\lfloor nt\rfloor}B_2(\{x\})s_{\gamma}''(x)dx$, $\{\cdot\}$ denotes the fractional part, and $B_2(\cdot)$ is the Bernoulli polynomial, i.e., $$B_2(x)=x^2-x+\frac{1}{6}.$$
Since $s''_{\gamma}(x)=\alpha(\alpha-1)(\gamma+x)^{\alpha-2}$ and $\alpha-2=2H-3<-1$, $R_2(\gamma,\lfloor nt\rfloor)$ is bounded for each fixed $0<\gamma<1$.
More precisely, let \begin{gather}R_2(\gamma)=\frac{\alpha(1-\alpha)}{2}\int_0^\infty B_2(\{x\})(\gamma+x)^{\alpha-2}dx,\\ \bar R_2(\lfloor nt\rfloor,\gamma)=\frac{\alpha(1-\alpha)}{2}\int_{\lfloor nt\rfloor}^\infty B_2(\{x\})(\gamma+x)^{\alpha-2}dx.
\end{gather}
Then,
    \begin{equation}
    \label{eq:r2}
        R_2(\lfloor nt\rfloor,\gamma)=R_2(\gamma)-\bar R_2(\lfloor nt\rfloor,\gamma).
    \end{equation}
We can obtain an upper bound of $\bar R_2$ due to the fact that $\sup_{[0,1]}|B_2(\cdot)|\leq 1/6$.
Notice that ,for $0<\gamma<1$,
\begin{gather}
\label{eq:r2bound}
\bar R_2(\lfloor nt\rfloor,\gamma)=\frac{\alpha(1-\alpha)}{2}\int_{\lfloor nt\rfloor}^\infty B_2(\{x\})(\gamma+x)^{\alpha-2}dx\leq\frac{\alpha}{12}(\lfloor nt\rfloor+\gamma)^{\alpha-1}.
\end{gather}
Combining \eqref{euler1} and \eqref{eq:r2}, we obtain, for $t\not\in\{j/n:j\in\mathbb N\}$,
\begin{equation}
\label{eq:fn1}
    \begin{aligned}
        f_n(t)&=\frac{1}{\alpha n^\alpha}\sum_{k=0}^{\lfloor nt\rfloor-1}s_{\{nt\}}(k)+\frac{1}{2\alpha}t^\alpha-\frac{n}{\alpha(\alpha+1)}t^{\alpha+1}\\
        &=-\frac{1}{\alpha(1+\alpha)n^{\alpha}}\{nt\}^{\alpha+1}+\frac{1}{2\alpha n^{\alpha}}\{nt\}^{\alpha}\\
        &\quad+\frac{1}{12n^{\alpha}}\left[(nt)^{\alpha-1}-\{nt\}^{\alpha-1}\right]+\frac{1}{\alpha n^{\alpha}}R_2(\{nt\})-\frac{1}{\alpha n^{\alpha}}\bar R_2(\lfloor nt\rfloor,\{nt\}).
    \end{aligned}
    \end{equation}
    Note that we can safely disregard the set $\{j/n:j\in\mathbb N\}$ since we only care about the integral of $f_n(t)$.
    It can be easily seen that $f_n(t)$ is integrable on each interval $[j/n,(j+1)/n]$, and, since $\alpha>0$, the integral in \eqref{eq:variance0} is bounded by $M_\alpha n^{-\alpha}$ with some constant $M_\alpha$ that only depends on $\alpha$ and therefore converges to $0$ as $n\to\infty$.
    The following result follows immediately.
\begin{lemma}
\label{lem:L2}
    As $n\to\infty$, $\mathcal P_n(t)\to0$ in $L^2$ and $\mathcal Q_n(t)\to0$ in $L^2$ for any $t\geq 0$.
\end{lemma}
    \subsubsection{Normalization and limiting variance \texorpdfstring{$\sigma_H^2$}{sigma H 2}}
    Since the integral in \eqref{eq:variance0} is of order $n^{-\alpha}$, we expect $n^{\alpha/2}\mathcal P_n(t)$ and $n^{\alpha/2}\mathcal Q_n(t)$ to have  non-trivial $L^2$ norm.
    The following result shows that $n^{\alpha}f_n(t)$ is close to a function that only depends on $\{nt\}$, the fractional part of $nt$, as $nt\to\infty$.
\begin{lemma}
\label{lem:fn}
For $0<\gamma<1$, define 
\begin{equation}
\label{def:Rgamma}
    R(\gamma)=-\frac{1}{\alpha(1+\alpha)}\gamma^{\alpha+1}+\frac{1}{2\alpha}\gamma^{\alpha}-\frac{1}{12}\gamma^{\alpha-1}+\frac{1}{\alpha}R_2(\gamma).
\end{equation}
    Then
    \begin{gather}
        |n^\alpha f_n(t)
        -R(\{nt\})|\leq \frac{1}{6} (nt)^{\alpha-1},
    \end{gather}for $t\not\in\{j/n:j\in\mathbb N\}$.
\end{lemma}
\begin{proof}
    Due to \eqref{eq:fn1}, we have that
    \begin{equation}
        \begin{aligned}
            n^\alpha f_n(t)&=-\frac{1}{\alpha(1+\alpha)}\{nt\}^{\alpha+1}+\frac{1}{2\alpha}\{nt\}^{\alpha}-\frac{1}{12}\{nt\}^{\alpha-1}\\
        &\quad+\frac{1}{\alpha}R_2(\{nt\})-\frac{1}{\alpha}\bar R_2(\lfloor nt\rfloor,\{nt\})+\frac{1}{12}(nt)^{\alpha-1}.
        \end{aligned}
    \end{equation}
    Due to \eqref{eq:r2bound}, we obtain the desired result.
\end{proof}
We present the next technical result that is useful later to prove that the normalized process $n^{\alpha/2}\mathcal P_n(t)$ has non-trivial variance.
\begin{lemma}
    \label{lem:sym_pos}
    Suppose $f(x)$ is a convex function on $[0,1]$.
    Then 
    \begin{equation}
        \int_0^1 B_2(x)f(x)dx\geq 0.
    \end{equation}
\end{lemma}
\begin{proof}
Let $x_1<x_2$ be the two zero points of $B_2(x)$.
Then 
\begin{align}
    \int_0^1 B_2(x)f(x)dx&=\int_0^{x_1}B_2(x)f(x)dx+\int_{x_1}^{x_2}B_2(x)f(x)dx+\int_{x_2}^1B_2(x)f(x)dx\\
    &=\int_0^{x_1}B_2(x)[f(x)+f(1-x)]dx+\int_{x_1}^{1/2}B_2(x)[f(x)+f(1-x)]dx\\
    &\geq \int_0^{x_1}B_2(x)[f(x_1)+f(x_2)]dx+\int_{x_1}^{1/2}B_2(x)[f(x_1)+f(x_2)]dx\\
    &=\frac{f(x_1)+f(x_2)}{2}\int_0^1 B_2(x)dx\\
    &=0,
\end{align}where the second line is due to the fact $B_2(x)=B_2(1-x)$; the third line follows from that $f(x)$ is convex, $B_2(\cdot)$ is positive on $(0,x_1)$ and negative on $(x_1,1/2)$; and the last line follows from the that $B_2(\cdot)$ has zero integral over $[0,1]$.
\end{proof} 

\begin{lemma}
\label{lem:variance}
    There exists a constant $\sigma_H>0$ that depends on $\alpha$ (equivalently, $H$) only such that, for each $t\in[0,1]$, $\E n^\alpha \mathcal P_n(t)^2\to\sigma_H^2t$ and $\E n^\alpha \mathcal Q_n(t)^2\to\sigma_H^2t$ as $n\to\infty$.
\end{lemma}
\begin{proof}
    We prove the result about $\mathcal P_n$, and the other result can be proved in the same way.
    Let $A(n,t)=n^\alpha f_n(t)
        -R(\{nt\})$.
    By Lemma~\ref{lem:fn}, we see that
    \begin{align}
        &\left|\E n^\alpha \mathcal P_n(t)^2-\alpha(\alpha+1) \int_0^t R(\{ns\})\left(\lceil ns\rceil-ns-\frac{1}{2}\right)ds\right|\\
        &=\alpha(\alpha+1)\left|\int_0^t A(n,s)\left(\lceil ns\rceil-ns-\frac{1}{2}\right)ds\right|\\
        &\leq\frac{\alpha(\alpha+1)}{2}\int_0^t|A(n,s)|ds\\
        &\leq\frac{\alpha+1}{12}n^{\alpha-1}t^\alpha\\
        &\leq \frac{\alpha+1}{12}n^{\alpha-1}.
    \end{align}
    On the other hand,
    \begin{align}
        &\int_0^{\lfloor nt\rfloor/n} R(\{ns\})\left(\lceil ns\rceil-ns-\frac{1}{2}\right)ds\\
        &=\sum_{k=1}^{\lfloor nt\rfloor}\int_{(k-1)/n}^{k/n}R(ns-k+1)\left(k-\frac{1}{2}-ns\right)ds\\
        &=\sum_{k=1}^{\lfloor nt\rfloor}\frac{1}{n}\int_{k-1}^k R(s-(k-1))\left(\frac{1}{2}-(s-(k-1))\right)ds\\
        &=\sum_{k=1}^{\lfloor nt\rfloor}\frac{1}{n}\int_{0}^1 R(\gamma)\left(\frac{1}{2}-\gamma\right)d\gamma\\
        &=\frac{\lfloor nt\rfloor}{n}\int_{0}^1 R(\gamma)\left(\frac{1}{2}-\gamma\right)d\gamma\\
        &\to \int_{0}^1 R(\gamma)\left(\frac{1}{2}-\gamma\right)d\gamma\cdot t,
    \end{align}
    and
    \begin{equation}
        \left|\int_{\lfloor nt\rfloor/n}^{t}R(\{ns\})\left(\lfloor nt\rfloor+\frac{1}{2}-ns\right)ds\right|\leq \frac{1}{2}\int_{\lfloor nt\rfloor/n}^{t}|R(\{ns\})|ds\leq \frac{1}{2n}\int_0^1|R(s)|ds\to0.
    \end{equation}
    It remains to show that 
    \begin{equation}
        \int_0^1 R(\gamma)\left(\frac{1}{2}-\gamma\right)d\gamma>0.
    \end{equation}
    Recall that
    \begin{equation}
        R(\gamma)=-\frac{1}{\alpha(1+\alpha)}\gamma^{\alpha+1}+\frac{1}{2\alpha}\gamma^{\alpha}-\frac{1}{12}\gamma^{\alpha-1}+\frac{1}{\alpha}R_2(\gamma).
    \end{equation}
    One can compute that
    \begin{gather}
        \int_0^1 \left[-\frac{1}{\alpha(1+\alpha)}\gamma^{\alpha+1}+\frac{1}{2\alpha}\gamma^{\alpha}-\frac{1}{12}\gamma^{\alpha-1}\right](\frac{1}{2}-\gamma)d\gamma=\frac{(3-\alpha)(1-\alpha)}{24\alpha(\alpha+1)(\alpha+3)}>0.
    \end{gather}
    It suffices to prove that 
    \begin{equation}
    \label{eq:positivity}
        \int_0^1 R_2(\gamma)\left(\frac{1}{2}-\gamma\right)d\gamma\geq0,
    \end{equation}which holds if $R_2(\cdot)$ is decreasing on $[0,1]$.
    Hence, the problem reduces to proving that
    \begin{equation}
        R_2'(\gamma)=-\frac{\alpha(1-\alpha)(2-\alpha)}{2}\int_0^\infty B_2(\{x\})(\gamma+x)^{\alpha-3}dx\leq0,
    \end{equation}
    which is further equivalent to 
    \begin{equation}
        \int_0^\infty B_2(\{x\})(\gamma+x)^{\alpha-3}dx\geq0.
    \end{equation}
    Since $0<\alpha<1$, $(\gamma+x)^{\alpha-3}$ is a convex function of $x$, the desired positivity \eqref{eq:positivity} follows from Lemma~\ref{lem:sym_pos}.
    Therefore, the result holds with \[\sigma_H^2=\alpha(\alpha+1)\int_{0}^1 R(\gamma)\left(\frac{1}{2}-\gamma\right)ds\geq\frac{(1-\alpha)(3-\alpha)}{24(\alpha+3)}>0.\qedhere\]
\end{proof}
In the sequel, we shall fix $\sigma_H$ as defined in the preceding lemma. With some simplification, it has the following explicit formula
\begin{equation}
\label{def:sigmaH}
\begin{aligned}
    \sigma_H^2 &= \frac{\alpha(\alpha+1)}{4}\int_0^\infty\left(\frac{2x^\alpha}{\alpha}-\frac{2(x+1)^\alpha}{\alpha}+x^{\alpha-1}+(x+1)^{\alpha-1}\right)\left(\{x\}^2-\{x\}+\frac{1}{6}\right)dx\\
    &\quad+\frac{(1-\alpha)(3-\alpha)}{24(\alpha+3)}.
\end{aligned}
\end{equation}
    \subsubsection{Weak convergence of the normalized integrands}
    The main result in this subsection is the following.
    \begin{lemma}
\label{lem:weak_convergence_one}
    As $n\to\infty$, $(n^{\alpha/2}\mathcal P_n(t),\mathcal B^H(t),n^{\alpha/2}\mathcal Q_n(t),\tilde{\mathcal B}^H(t))\to(\sigma_H W(t), \mathcal B^H(t),\sigma_H\tilde W(t), \tilde{\mathcal B}^H(t))$ weakly in $C^\theta([0,1])$, where $W(t)$ is an independent Brownian motion, and $\tilde W(t)=W(1-t)-W(1)$.
\end{lemma}
The proof of this lemma consists of two parts: (1) showing convergence of the covariance of the Gaussian processes to the desired limit (Lemmas~\ref{flem:covariance} and \ref{flem:independence}), and (2) establishing tightness of the family of pre-limit processes (Lemmas~\ref{flem:Kolmogorov} and \ref{flem:tight}).
    \begin{lemma}
    \label{flem:covariance}
    For all $s,t\in[0,1]$, as $n\to\infty$, we have
    \begin{equation}
        \E n^\alpha \mathcal P_n(s)\mathcal P_n(t)\to \sigma_H^2(s\wedge t),\quad\E n^\alpha\mathcal Q_n(s)\mathcal Q_n(t)\to \sigma_H^2(s\wedge t).
    \end{equation}
\end{lemma}
\begin{proof}
    We only prove the first claim in the result as the second one can be proved in the same way.
    Without loss of generality, we assume $s\leq t$.
    Due to Lemma~\ref{lem:variance}, it suffices to prove that
    \begin{align}
        \E [n^\alpha(\mathcal P_n(t)-\mathcal P_n(s))\mathcal P_n(s)]&=\int_s^t\int_0^s n^\alpha(u-r)^{\alpha-1}\left(\lceil nr\rceil-nr-\frac{1}{2}\right)dr\cdot \left(\lceil nu\rceil-nu-\frac{1}{2}\right)du\\
        &=:\int_s^t g_n(u)\cdot \left(\lceil nu\rceil-nu-\frac{1}{2}\right)du\to0.
    \end{align}
    As before, we invoke the Euler-Maclaurin summation formula to explicitly compute $g_n(u)$ (here $B_2$ and $R_2$ denote corresponding Bernoulli number and remainder):
    \begin{align}
        &g_n(u,s)\\
        &=n^\alpha\int_0^s(u-r)^{\alpha-1}\left(\lceil nr\rceil-nr-\frac{1}{2}\right)dr\\
        &=n^\alpha\int_0^{\lfloor ns\rfloor/n}(u-r)^{\alpha-1}\left(\lceil nr\rceil-nr-\frac{1}{2}\right)dr+n^\alpha\int_{\lfloor ns\rfloor/n}^s(u-r)^{\alpha-1}\left(\lceil nr\rceil-nr-\frac{1}{2}\right)dr\\
        &=n^\alpha\sum_{k=1}^{\lfloor ns\rfloor}\int_{(k-1)/n}^{k/n}(u-r)^{\alpha-1}\left(\lceil nr\rceil-nr-\frac{1}{2}\right)dr+n^\alpha\int_{\lfloor ns\rfloor/n}^s(u-r)^{\alpha-1}\left(\lceil nr\rceil-nr-\frac{1}{2}\right)dr\\
        &=\frac{n^{\alpha+1}}{\alpha+1}\sum_{k=1}^{\lfloor ns\rfloor}\left[\left(u-\frac{k-1}{n}\right)^{\alpha+1}-\left(u-\frac{k}{n}\right)^{\alpha+1}\right]\\
        &\quad+\frac{n^\alpha}{\alpha}\sum_{k=1}^{\lfloor ns\rfloor}\left(k-\frac{1}{2}-nu\right)\left[\left(u-\frac{k-1}{n}\right)^{\alpha}-\left(u-\frac{k}{n}\right)^{\alpha}\right]\\
        &\quad+\frac{n^{\alpha+1}}{\alpha+1}\left[\left(u-\frac{\lfloor ns\rfloor}{n}\right)^{\alpha+1}-\left(u-s\right)^{\alpha+1}\right]\\
        &\quad+\left(\lceil ns\rceil-\frac{1}{2}-nu\right)\cdot\frac{n^\alpha}{\alpha}\left[\left(u-\frac{\lfloor ns\rfloor}{n}\right)^{\alpha}-\left(u-s\right)^{\alpha}\right]\\
        &=\frac{1}{\alpha}\sum_{k=1}^{\lfloor ns\rfloor}\left(nu-k\right)^\alpha-\frac{n^{\alpha+1}}{\alpha(\alpha+1)}(u^{\alpha+1}-(u-s)^{\alpha+1})+\frac{n^\alpha}{2\alpha}(u^\alpha+(u-s)^\alpha)\\
        &\quad-\frac{n^\alpha}{\alpha}(\lceil ns\rceil -ns)(u-s)^\alpha\\
        &=\frac{1}{\alpha(\alpha+1)}\left[(nu)^{\alpha+1}-(nu-\lfloor ns\rfloor)^{\alpha+1}\right]+\frac{1}{2\alpha}\left[(nu-\lfloor ns\rfloor)^\alpha-(nu)^\alpha\right]\\
        &\quad-\frac{n^{\alpha+1}}{\alpha(\alpha+1)}(u^{\alpha+1}-(u-s)^{\alpha+1})+\frac{n^\alpha}{2\alpha}(u^\alpha+(u-s)^\alpha)\\
        &\quad+\frac{B_2}{2}\left[(nu)^{\alpha-1}-(nu-\lfloor ns\rfloor)^{\alpha-1}\right]+R_2/\alpha-\frac{n^\alpha}{\alpha}(\lceil ns\rceil -ns)(u-s)^\alpha\\
        &=\frac{1}{\alpha(\alpha+1)}\left[(nu-ns)^{\alpha+1}-(nu-\lfloor ns\rfloor)^{\alpha+1}\right]+\frac{1}{2\alpha}\left[(nu-\lfloor ns\rfloor)^\alpha+(nu-ns)^\alpha\right]\\
        &\quad+\frac{B_2}{2}\left[(nu)^{\alpha-1}-(nu-\lfloor ns\rfloor)^{\alpha-1}\right]+R_2/\alpha-\frac{n^\alpha}{\alpha}(\lceil ns\rceil -ns)(u-s)^\alpha\\
        &=-\frac{1}{\alpha}(nu-ns)^\alpha(ns-\lfloor ns\rfloor)+\frac{1}{\alpha}(nu-ns)^\alpha-\frac{n^\alpha}{\alpha}(\lceil ns\rceil -ns)(u-s)^\alpha\\
        &\quad+\frac{B_2}{2}\left[(nu)^{\alpha-1}-(nu-\lfloor ns\rfloor)^{\alpha-1}\right]+R_2/\alpha+O((nu-ns)^{\alpha-1})\\
        &=\frac{B_2}{2}\left[(nu)^{\alpha-1}-(nu-\lfloor ns\rfloor)^{\alpha-1}\right]+R_2/\alpha+O((nu-ns)^{\alpha-1}).
    \end{align}
    Since the remainder $R_2$ is bounded (up to a constant) by  $\left|(nu)^{\alpha-1}-(nu-\lfloor ns\rfloor)^{\alpha-1}\right|$, $|g_n(u,s)|\lesssim (nu)^{\alpha-1}+(nu-ns)^{\alpha-1}$.
    The desired result follows by explicit integration.
\end{proof}

\begin{lemma}
\label{flem:independence}
    For all $s,t\in[0,1]$, as $n\to\infty$, we have
    \begin{equation}
        \E n^\alpha\mathcal P_n(s)\mathcal B^H(t)\to 0.
    \end{equation}
\end{lemma}
\begin{proof}
    It suffices to prove that
    \begin{equation}
        n^\alpha\int_0^t\int_0^s|u-r|^{\alpha-1}\left(\lceil nr\rceil-nr-\frac{1}{2}\right)drdu\to0\text{, as }n\to\infty.
    \end{equation}
    If $s\leq t$,
    \begin{align}
        &n^\alpha\int_0^t\int_0^s|u-r|^{\alpha-1}\left(\lceil nr\rceil-nr-\frac{1}{2}\right)drdu\\
        &=n^\alpha\int_0^s\left(\lceil nr\rceil-nr-\frac{1}{2}\right)\int_0^t|u-r|^{\alpha-1}dudr\\
        &=\frac{n^\alpha}{\alpha}\int_0^s\left(\lceil nr\rceil-nr-\frac{1}{2}\right)(r^\alpha+(s-r)^\alpha)dr\lesssim n^{\alpha-1}\to0;
    \end{align}
    if $s>t$,
    \begin{align}
        &n^\alpha\int_0^t\int_0^s|u-r|^{\alpha-1}\left(\lceil nr\rceil-nr-\frac{1}{2}\right)drdu\\
        &=n^\alpha\int_0^s\left(\lceil nr\rceil-nr-\frac{1}{2}\right)\int_0^t|u-r|^{\alpha-1}dudr\\
        &=n^\alpha\int_0^t\left(\lceil nr\rceil-nr-\frac{1}{2}\right)\int_0^t|u-r|^{\alpha-1}dudr\\
        &\quad+n^\alpha\int_t^s\left(\lceil nr\rceil-nr-\frac{1}{2}\right)\int_0^t|u-r|^{\alpha-1}dudr\\
        &=\frac{n^\alpha}{\alpha}\int_0^s\left(\lceil nr\rceil-nr-\frac{1}{2}\right)(r^\alpha+(s-r)^\alpha)dr\\
        &\quad+\frac{n^\alpha}{\alpha}\int_t^s\left(\lceil nr\rceil-nr-\frac{1}{2}\right)(r^\alpha-(r-t)^\alpha)dr\\
        &\lesssim n^{\alpha-1}\to0.\qedhere
    \end{align}
\end{proof}

\begin{lemma}
\label{flem:Kolmogorov}
    For each integer $p\geq1$, there exists a constant $C_{\alpha,p}$ such that, for all $n\geq1$, $0\leq s\leq t\leq1$,
    \begin{equation}
        \E \left|n^{\alpha/2}(\mathcal P_n(t)-\mathcal P_n(s))\right|^{2p}\leq C_{\alpha,p} |t-s|^p.
    \end{equation}
\end{lemma}
\begin{proof}
    Since the integrand is deterministic, the Stratonovich integral is the same as the divergence integral, i.e., 
    \begin{equation}
        \mathcal P_n(t)-\mathcal P_n(s)=\int_s^t\left(\lceil nu\rceil-nu-\frac{1}{2}\right)\circ dB^H(u) = \delta\left(\mathbf{1}_{[s,t]}\left(\lceil nu\rceil-nu-\frac{1}{2}\right)\right).
    \end{equation}
    By Proposition 1.5.8 in \cite{MR2200233}, we know that, for each integer $p\geq1$, there exists a constant $C_p$ such that
    \begin{equation}
    \label{eq:2p_norm}
        \E[(\mathcal P_n(t)-\mathcal P_n(s))^{2p}]=\E \left[\delta\left(\mathbf{1}_{[s,t]}\left(\lceil nu\rceil-nu-\frac{1}{2}\right)\right)^{2p}\right]\leq C_p\left\|\mathbf{1}_{[s,t]}\left(\lceil nu\rceil-nu-\frac{1}{2}\right)\right\|_{\mathcal H}^{2p}.
    \end{equation}
    We proceed to compute explicitly the norm on the right-hand side. As in Subsection~\ref{sec:before_norm}, more specifically, in \eqref{def:fn}, we define function
    \begin{equation}
        f_n(s,u)=\int_s^u (u-r)^{\alpha-1}\left(\lceil nr\rceil-nr-\frac{1}{2}\right)dr.
    \end{equation}
    In particular, $f_n(0,u)$ coincides with $f_n(u)$ defined in \eqref{def:fn}.
    One can compute the integral explicitly. 
    Namely, for $s,u\not\in\{j/n:j\in\mathbb N\}$ such that there does not exist $j\in\mathbb N$ satisfying $s<j/n<u$,
    \begin{align}
        f_n(s,u)&=\frac{\lceil nu\rceil-nu-1/2}{\alpha}(u-s)^\alpha+\frac{n}{\alpha+1}(u-s)^{\alpha+1}.\label{eq:fn_nondivided}
    \end{align}
    For $s,u\not\in\{j/n:j\in\mathbb N\}$ such that there exists $j\in\mathbb N$ satisfying $s<j/n<u$, by explicit integration and the Euler-Maclaurin summation formula,
    \begin{align}
        f_n(s,u)&=\frac{1}{\alpha}\sum_{k=\lceil ns\rceil}^{\lfloor nu\rfloor }(u-k/n)^\alpha-\frac{n}{\alpha(\alpha+1)}(u-s)^{\alpha+1}+\frac{\lceil ns\rceil-ns-1/2}{\alpha}(u-s)^\alpha\label{eq:fn_divided}\\
        &=\frac{1}{\alpha(\alpha+1)n^\alpha}\left[(nu-\lceil ns\rceil)^{\alpha+1}-(nu-ns)^{\alpha+1}\right]+\frac{\lceil ns\rceil-ns}{\alpha n^\alpha}(nu-ns)^\alpha\\
        &\quad+\frac{1}{2\alpha n^\alpha}[(nu-\lceil ns\rceil)^\alpha-(nu-ns)^\alpha]-\frac{1}{\alpha(\alpha+1)n^\alpha}\{nu\}^{\alpha+1}\\
        &\quad+\frac{1}{2\alpha n^\alpha}\{nu\}^{\alpha}-\frac{1}{12n^\alpha}[\{nu\}^{\alpha-1}-(nu-\lceil ns\rceil)^{\alpha-1}]+\frac{R_2}{\alpha n^\alpha},
    \end{align}where $R_2$ is bounded by $\frac{\alpha}{12}[(nu-\lfloor ns\rfloor)^{\alpha-1}-\{nu\}^{\alpha-1}]$.
    Then, by Taylor's expansion, there exists a constant $C_\alpha$ such that
    \begin{equation}
    \label{eq:fn_far}
        |f_n(s,u)|\leq\frac{C_\alpha^1}{n^\alpha}((nu-\lceil ns\rceil)^{\alpha-1}+\{nu\}^{\alpha-1}+1).
    \end{equation}
    \begin{enumerate}
    \item For $0\leq s<t\leq1$ such that there does not exist $j\in\mathbb N$ satisfying $s<j/n<u$, by \eqref{eq:fn_nondivided},
    we have
    \begin{equation}
    \label{eq:holder-nondivided}
    \begin{aligned}
        &\left|\int_s^tn^\alpha f_n(s,u)\left(\lceil nu\rceil-nu-\frac{1}{2}\right)du\right|\\
        &\leq\frac{n^\alpha}{\alpha}\int_s^t(u-s)^\alpha ds+\frac{n^{1+\alpha}}{\alpha+1}\int_s^t(u-s)^{\alpha+1}ds\\
        &=\frac{n^\alpha}{\alpha(\alpha+1)}(t-s)^{\alpha+1}+\frac{n^{\alpha+1}}{(\alpha+1)(\alpha+2)}(t-s)^{\alpha+2}\\
        &\leq C_\alpha^2(t-s),
    \end{aligned}
    \end{equation}since $|t-s|\leq 1/n$.
    \item For $0\leq s<t\leq1$ such that there exists exactly one $j\in\mathbb N$ satisfying $s<j/n<t$, by \eqref{eq:fn_divided},
    we have
    \begin{align}
        &\left|\int_{j/n}^t n^\alpha f_n(s,u)\left(\lceil nu\rceil-nu-\frac{1}{2}\right)du\right|\\
        &=\left|\int_{j/n}^tn^\alpha\left[\frac{1}{\alpha}(u-j/n)^\alpha-\frac{n}{\alpha(\alpha+1)}(u-s)^{\alpha+1}+\frac{j-ns-1/2}{\alpha}(u-s)^\alpha\right]\left(\lceil nu\rceil-nu-\frac{1}{2}\right)du\right|\\
        &\leq \int_{j/n}^t\left[\frac{1}{\alpha}+\frac{2^{\alpha+1}}{\alpha(\alpha+1)}+\frac{2^\alpha}{\alpha}\right] du\\
        &\leq C_\alpha^3(t-j/n),
    \end{align}since $u-j/n<1/n$ and $u-s<2/n$.
    Therefore, combining with \eqref{eq:holder-nondivided}, we have
    \begin{equation}
    \label{eq:holder_divided}
        \left|\int_s^tn^\alpha f_n(s,u)\left(\lceil nu\rceil-nu-\frac{1}{2}\right)du\right|\leq C_\alpha^4(t-s).
    \end{equation}
    \item For $0\leq s<t\leq1$ such that there exist more than one $j\in\mathbb N$ satisfying $s<j/n<u$, by \eqref{eq:fn_far} and \eqref{eq:holder_divided},
        \begin{align*}
            &\left|\int_s^tn^\alpha f_n(s,u)\left(\lceil nu\rceil-nu-\frac{1}{2}\right)du\right|\\
            &\leq \left|\int_s^{(\lceil ns\rceil+1)/n}n^\alpha f_n(s,u)\left(\lceil nu\rceil-nu-\frac{1}{2}\right)du\right|+\left|\int_{(\lceil ns\rceil+1)/n}^tn^\alpha f_n(s,u)\left(\lceil nu\rceil-nu-\frac{1}{2}\right)du\right|\\
            &\leq C_\alpha^4({(\lceil ns\rceil+1)/n}-s)+C_\alpha^1\int_{(\lceil ns\rceil+1)/n}^t[(nu-\lceil ns\rceil)^{\alpha-1}+\{nu\}^{\alpha-1}+1]du\\
            &\leq (C_\alpha^1+C_\alpha^4)(t-s)+C_\alpha^1\int_{(\lceil ns\rceil+1)/n}^t[(nu-\lceil ns\rceil)^{\alpha-1}+\{nu\}^{\alpha-1}]du\\
            &\leq (2C_\alpha^1/\alpha+C_\alpha^4)(t-s)+\frac{C_\alpha^1}{\alpha} n^{\alpha-1}[(t-\lceil ns\rceil/n)^{\alpha}-n^{-\alpha}]+\frac{C_\alpha^1}{\alpha}\frac{\lfloor nt\rfloor-\lceil ns\rceil-1}{n}+\frac{\{nt\}^\alpha}{\alpha n}\\
            &\leq C_\alpha^5(t-s),
        \end{align*}since $t-s\geq t-\lceil ns\rceil/n\geq 1/n$.
    \end{enumerate}
    We proved that there exists a constant $C_\alpha$ such that for all $n\geq 1$, $0\leq s\leq t\leq1$,
    \begin{equation}
        \left\|n^{\alpha/2}\mathbf{1}_{[s,t]}\left(\lceil nu\rceil-nu-\frac{1}{2}\right)\right\|_{\mathcal H}^2=\frac{\alpha(\alpha+1)}{2}\left|\int_s^tn^\alpha f_n(s,u)\left(\lceil nu\rceil-nu-\frac{1}{2}\right)du\right|\leq C_\alpha(t-s).
    \end{equation}
    By \eqref{eq:2p_norm}, we now have the desired result
    \begin{equation*}
        \E [(n^{\alpha/2}(\mathcal P_n(t)-\mathcal P_n(s)))^{2p}]\leq C_p \left\|n^{\alpha/2}\mathbf{1}_{[s,t]}\left(\lceil nu\rceil-nu-\frac{1}{2}\right)\right\|_{\mathcal H}^{2p}\leq C_p(C_\alpha)^p(t-s)^p.\qedhere
    \end{equation*}
\end{proof}

\begin{lemma}
\label{flem:tight}
For any $\theta\in(0,1/2)$, the family of the processes $\{n^{\alpha/2}\mathcal P_n,n\geq1\}$ is tight in the space of $C^\theta([0,1])$.
\end{lemma}

\begin{proof}
    Take $p$ large such that $(p-1)/2p>\theta$ and then apply Kolmogorov's continuity theorem with Lemma~\ref{flem:Kolmogorov}.
\end{proof}

\begin{proof}[Proof of Lemma~\ref{lem:weak_convergence_one}]
    The tightness of the family of the processes $(n^{\alpha/2}\mathcal P_n(t),\mathcal B^H(t),n^{\alpha/2}\mathcal Q_n(t),\tilde{\mathcal B}^H(t))$ is due to Lemma~\ref{flem:tight}, and the processes are Gaussian processes, so it suffices to prove the convergence of covariance matrices to the desired covariance matrix.
    \begin{enumerate}[(1)]
        \item The convergence of $\E n^\alpha \mathcal P_n(s)\mathcal P_n(t)$ and $\E n^\alpha \mathcal Q_n(s)\mathcal Q_n(t)$ are due to Lemma~\ref{flem:covariance};
        \item The convergence of $\E n^{\alpha/2}\mathcal P_n(s)\mathcal B^H(t)$ and $\E n^{\alpha/2}\mathcal Q_n(s)\tilde{\mathcal B}^H(t)$ are due to Lemma~\ref{flem:independence};
        \item The convergence of $\E n^{\alpha/2}\mathcal P_n(s)\tilde B^H(t)$ and $\E n^{\alpha/2}\mathcal Q_n(s){\mathcal B}^H(t)$ can also be easily proved using Lemma~\ref{flem:independence};
        \item Lastly, we prove the convergence of $n^\alpha\E \mathcal P_n(s)\mathcal Q_n(t)$. By Lemma~\ref{flem:covariance}, we have
        \begin{align}
            \E n^\alpha\mathcal  P_n(s)\mathcal Q_n(t)&= -n^\alpha\int_{1-t}^1\int_0^s|u-r|^{\alpha-1}\left(\lceil nr\rceil-nr-\frac{1}{2}\right)\left(\lceil nu\rceil-nu-\frac{1}{2}\right)dudr\\
            &=n^\alpha\int_{0}^{1-t}\int_0^s|u-r|^{\alpha-1}\left(\lceil nr\rceil-nr-\frac{1}{2}\right)\left(\lceil nu\rceil-nu-\frac{1}{2}\right)dudr\\
            &\quad-n^\alpha\int_0^1\int_0^s|u-r|^{\alpha-1}\left(\lceil nr\rceil-nr-\frac{1}{2}\right)\left(\lceil nu\rceil-nu-\frac{1}{2}\right)dudr\\
            &=\E n^\alpha\mathcal P_n(1-t)\mathcal P_n(s)-\E n^\alpha\mathcal P_n(1)\mathcal P_n(s)\\
            &\to \sigma_H^2((1-t)\wedge s-s)=\sigma_H^2(1-t-s)\wedge 0.\qedhere
        \end{align}
        
    \end{enumerate}
\end{proof}
\subsection{Weak convergence of the covariance and correlation: Proof of Theorem~\ref{fthm:cov}}
The proofs in Subsections~\ref{fsec:indep} and \ref{fsec:dep} are similar in spirit to those in Subsections~\ref{ssec:indep} and \ref{ssec:dep}, while the proof of Proposition~\ref{fprop:conj7} in Subsection~\ref{fsec:corr} involving delta method requires additional care due to the additional multiplier $n^{H-\frac{1}{2}}$.
\subsubsection{Covariance: independence case}
\label{fsec:indep}
    Suppose that we have two independent fractional Brownian motions $B_1^H$ and $B_2^H$ with the same Hurst parameter $H>1/2$. Recall that we defined $\alpha = 2H-1$ and 
\begin{align}
    B_2(x)&=x^2-x+\frac{1}{6},\\
    R_2(\gamma)&=\frac{\alpha(1-\alpha)}{2}\int_0^\infty B_2(\{x\})(\gamma+x)^{\alpha-2}dx,\\
    R(\gamma)&=-\frac{1}{\alpha(1+\alpha)}\gamma^{\alpha+1}+\frac{1}{2\alpha}\gamma^{\alpha}-\frac{1}{12}\gamma^{\alpha-1}+\frac{1}{\alpha}R_2(\gamma),\\
    \sigma_H^2&=\alpha(\alpha+1)\int_{0}^1 R(\gamma)(\frac{1}{2}-\gamma)ds>0.
\end{align}
Let us further define
\begin{align}
    P_n(t)&=n^{\alpha/2}\int_0^t\left(\lceil ns\rceil-ns-\frac{1}{2}\right)\circ dB^H_1(s),\\
    Q_n(t)&=-n^{\alpha/2}\int_{1-t}^1\left(\lceil ns\rceil-ns-\frac{1}{2}\right)\circ dB^H_1(s),\\
    R_n(t)&=n^{\alpha/2}\int_0^t\left(\lceil ns\rceil-ns-\frac{1}{2}\right)\circ dB^H_2(s),\\
    S_n(t)&=-n^{\alpha/2}\int_{1-t}^1\left(\lceil ns\rceil-ns-\frac{1}{2}\right)\circ dB_2^H(s).
\end{align}
As in Lemma~\ref{lem:weak_convergence_one}, one can obtain the following result,
\begin{lemma}
\label{lem:weak_convergence_two}
For any $0<\theta<1/2$, as $n\to\infty$,
\begin{equation}
    (P_n(t),Q_n(t),R_n(t),S_n(t),B^H_1(t),\tilde B^H_1(t),B_2^H(t),\tilde B_2^H(t))
\end{equation}
converges weakly in $C^\theta([0,1])$ to 
\begin{equation}
    \left(\sigma_H W_1(t),\sigma_H \tilde W_1(t),\sigma_H W_1(t),\sigma_H \tilde W_1(t),B^H_1(t),\tilde B^H_1(t),B_2^H(t),\tilde B_2^H(t)\right),
\end{equation}
where $W_1,W_2$ are exogenous independent Brownian motions, $\tilde B^H_i(t)=B_i^H(1-t)-B_i^H(1)$, and $\tilde W_i(t)=W_i(1-t)-W_i(1)$, $i=1,2$.
\end{lemma}

Recall the functionals defined in \eqref{def:functionals}. For simplicity of notation, we define
\begin{gather}
\label{def:X^11}
    X^{i,i}=D(B^H_i),\quad X_{n}^{i,i}=D_n(B^H_i),\quad\text{for }i=1, 2,\\
    X^{1,2}=A(B^H_1,B^H_2),\quad X^{1,2}_n=A_n(B^H_1,B^H_2).
\end{gather}
\begin{lemma}
    \label{lem:convergence_integral}
    Let $U=(X^{1,2},X^{1,1},X^{2,2})^{\mathrm{T}}$ and $U_n=(X^{1,2}_n,X^{1,1}_n,X^{2,2}_n)^{\mathrm{T}}$. If $B^H_1$ and $B^H_2$ are independent, then $n(U_n-U)=V+n^{-\alpha/2}Z_n$ and $(Z_n,B^H_1,B^H_2)\to (Z,B^H_1,B^H_2)$ in distribution, where
    \begin{equation}
    \begin{aligned}
        V&=\begin{pmatrix}
            \frac{1}{2}B^H_1(1)\bar B_2^H+\frac{1}{2}B^H_2(1)\bar B_1^H-\frac{1}{2}B^H_1(1)B^H_2(1)\\
            B^H_1(1)\bar B_1^H-\frac{1}{2}B^H_1(1)^2\\
            B^H_2(1)\bar B_2^H-\frac{1}{2}B^H_2(1)^2
        \end{pmatrix},
    \end{aligned}
    \end{equation}
    \begin{equation}
        Z=\sigma_H
        \begin{pmatrix}
            \int_0^1 [\bar B_2^H-B_2^H(t)]dW_1(t)+\int_0^1 [\bar B_1^H-B_1^H(t)]dW_2(t)\\
            2\int_0^1 [\bar B_1^H-B_1^H(t)]dW_1(t)\\
            2\int_0^1 [\bar B_2^H-B_2^H(t)]dW_2(t)
        \end{pmatrix},
    \end{equation}where $\bar B_i^H=\int_0^1 B_i^H(t)dt$ and $W_i$ are independent Brownian motions.
    Moreover, conditionally on $B_1^H$ and $B_2^H$, $Z\sim\mathcal N(0,\Sigma)$, where
    \begin{equation}
        \Sigma=\sigma_H^2\begin{pmatrix}
            X^{1,1}+X^{2,2}&2X^{1,2}& 2X^{1,2}\\
            2X^{1,2}&4X^{1,1}&0\\
            2X^{1,2}&0& 4X^{2,2}
        \end{pmatrix}.
    \end{equation}
\end{lemma}
\begin{proof}
As in Subsection~\ref{fsec:representation}, we can represent $n(X_n^{1,1}-X^{1,1})$ and $n(X_n^{2,2}-X^{2,2})$ as Stratonovich integrals. Namely,
\begin{align}
    n(X_n^{1,1}-X^{1,1})&=\int_0^1(B_1^H(t)-tB_1^H(1))\circ dB_1^H(t)\\
    &\quad+2\int_0^1(1-t)P_n(t)\circ dB_1^H(t)-2\int_0^1(1-t) Q_n(t)\circ d\tilde B_1^H(t)+E^{1,1}_n\\
    &=B^H_1(1)\bar B_1^H-\frac{1}{2}B^H_1(1)^2+2\int_0^1[(1-t)P_n(t)+tQ_n(1-t)]\circ dB_1^H(t)+E^{1,1}_n\\
    n(X_n^{2,2}-X^{2,2})&=\int_0^1(B_2^H(t)-tB_2^H(1))\circ dB_2^H(t)\\
    &\quad+2\int_0^1(1-t)R_n(t)\circ dB_2^H(t)-2\int_0^1(1-t) S_n(t)\circ d\tilde B_2^H(t)+E^{2,2}_n\\
    &=B^H_2(1)\bar B_2^H-\frac{1}{2}B^H_2(1)^2+2\int_0^1[(1-t)R_n(t)+tS_n(1-t)]\circ dB_2^H(t)+E^{2,2}_n,
\end{align}
where
\begin{equation}
    E^{i,i}_n=-2\int_0^1\int_0^t\frac{1}{n}(\lceil nt\rceil-nt)(\lceil ns\rceil-ns)\circ dB_i^H(s)\circ dB_i^H(t).
\end{equation}
Similarly, as in Subsection~\ref{fsec:representation}, one can also represent $n(X_n^{1,2}-X^{1,2})$ as Stratonovich integral. Namely,
\begin{align}
    n(X_n^{1,2}-X^{1,2})&=\frac{1}{2}\int_0^1(B_1^H(t)-tB_1^H(1))\circ dB_2^H(t)+\frac{1}{2}\int_0^1(B_2^H(t)-tB_2^H(1))\circ dB_1^H(t)\\
    &\quad+\int_0^1(1-t)P_n(t)\circ dB_2^H(t)-\int_0^1(1-t) Q_n(t)\circ d\tilde B_2^H(t)\\
    &\quad+\int_0^1(1-t)R_n(t)\circ dB_1^H(t)-\int_0^1(1-t) S_n(t)\circ d\tilde B_1^H(t)+E^{1,2}_n\\
    &=\frac{1}{2}B^H_1(1)\bar B_2^H+\frac{1}{2}B^H_2(1)\bar B_1^H-\frac{1}{2}B^H_1(1)B^H_2(1)\\
    &\quad+\int_0^1[(1-t)P_n(t)+tQ_n(1-t)]\circ dB_2^H(t)\\
    &\quad+\int_0^1[(1-t)R_n(t)+tS_n(1-t)]\circ dB_1^H(t)+E^{1,2}_n,
\end{align}where
\begin{equation}
    E^{1,2}_n=\int_0^1\int_0^1\frac{1}{n}(\lceil nt\rceil-nt)(\lceil ns\rceil-ns)\circ dB_1^H(s)\circ dB_2^H(t).
\end{equation}
    Choose $1/2<\beta<H$ and $1-\beta<\theta<1/2$. By Lemma~\ref{lem:weak_convergence_two}, we have that 
    \begin{equation}
    (P_n(t),Q_n(t),R_n(t),S_n(t),B^H_1(t),\tilde B^H_1(t),B_2^H(t),\tilde B_2^H(t))
\end{equation}
converges weakly in $C^\theta([0,1])$ to 
\begin{equation}
    \left(\sigma_H W_1(t),\sigma_H \tilde W_1(t),\sigma_H W_1(t),\sigma_H \tilde W_1(t),B^H_1(t),\tilde B^H_1(t),B_2^H(t),\tilde B_2^H(t)\right).
\end{equation}
Since the Stratonovich integral (Young integral) is continuous from $C^\theta\times C^\beta$ to $C^\beta$, by Skorokhod's representation theorem, we conclude that 
    \begin{align*}
        Z_n&=n^{\alpha/2}
        \begin{pmatrix}
            \int_0^1[(1-t)P_n(t)+tQ_n(1-t)]\circ dB_2^H(t)+\int_0^1[(1-t)R_n(t)+tS_n(1-t)]\circ dB_1^H(t)\\
            2\int_0^1[(1-t)P_n(t)+tQ_n(1-t)]\circ dB_1^H(t)\\
            2\int_0^1[(1-t)R_n(t)+tS_n(1-t)]\circ dB_2^H(t)
        \end{pmatrix}\\
        &\quad+n^{\alpha/2}
        \begin{pmatrix}
            E_n^{1,2}\\
            E_n^{1,1}\\
            E_n^{2,2}
        \end{pmatrix}\\
        &\to\sigma_H
        \begin{pmatrix}
            \int_0^1 [W_1(t)-tW_1(1)]\circ dB_2^H(t)+\int_0^1 [W_2(t)-tW_2(1)]\circ dB_1^H(t)\\
            2\int_0^1 [W_1(t)-tW_1(1)]\circ dB_1^H(t)\\
            2\int_0^1 [W_2(t)-tW_2(1)]\circ dB_2^H(t)
        \end{pmatrix}\\
        &=\sigma_H
        \begin{pmatrix}
            \int_0^1 [\bar B_2^H-B_2^H(t)]dW_1(t)+\int_0^1 [\bar B_1^H-B_1^H(t)]dW_2(t)\\
            2\int_0^1 [\bar B_1^H-B_1^H(t)]dW_1(t)\\
            2\int_0^1 [\bar B_2^H-B_2^H(t)]dW_2(t)
        \end{pmatrix}.\qedhere
    \end{align*}
\end{proof}
\subsubsection{Covariance: general case}
\label{fsec:dep}
    Now we are ready to consider the general case we originally have in Theorem~\ref{fthm:cov}.
    Since we have two fractional Brownian motions $\mathcal B_1^H$ and $\mathcal B_2^H$ that are jointly Gaussian with constant correlation coefficient $r\in(-1,1)$, $\mathcal B_1$ and $(\mathcal B_2^H-r\mathcal B_1^H)/\sqrt{1-r^2}$ are two independent fractional Brownian motions with the same Hurst parameter $H$. Then, Lemma~\ref{lem:convergence_integral} can be applied to $B_1^H=\mathcal B_1^H$ and $B_2^H=(\mathcal B_2^H-r\mathcal B_1^H)/\sqrt{1-r^2}$.
    Recalling the notation in \eqref{def:empirical}, the empirical covariance of $\mathcal B_1^H$ and $\mathcal B_2^H$ can be represented as
    \begin{gather}
    \mathcal X^{i,i}=D(\mathcal B_i^H),\quad \mathcal X_{n}^{i,i}=D_n(\mathcal B_i^H),\quad\text{for }i=1, 2,\\
    \mathcal X^{1,2}=A(\mathcal B_1^H,\mathcal B_2^H),\quad \mathcal X^{1,2}_n=A_n(\mathcal B_1^H,\mathcal B_2^H).
\end{gather}
    Recalling the notation in \eqref{def:X^11}, by linearity of integration and summation, we obtain that
    \begin{equation}
\label{feq:linearity}
    \begin{pmatrix}
        \mathcal X_n^{1,2}-\mathcal X^{1,2}\\
        \mathcal X_n^{1,1}-\mathcal X^{1,1}\\
        \mathcal X_n^{2,2}-\mathcal X^{2,2}
    \end{pmatrix}=\begin{pmatrix}
        \sqrt{1-r^2}&r&0\\
        0&1&0\\
        2r\sqrt{1-r^2}&r^2&1-r^2
    \end{pmatrix}\cdot
    \begin{pmatrix}
        X_n^{1,2}-X^{1,2}\\
        X_n^{1,1}-X^{1,1}\\
        X_n^{2,2}-X^{2,2}
    \end{pmatrix}=:A_r(U_n-U).
\end{equation}

\begin{proposition}
    \label{fprop:conj6}
    Let $\mathcal U=(\mathcal X^{1,2},\mathcal X^{1,1},\mathcal X^{2,2})$ and $\mathcal U_n=(\mathcal X^{1,2}_n,\mathcal X^{1,1}_n,\mathcal X^{2,2}_n)$. If $\mathcal B^H_1$ and $\mathcal B^H_2$ are jointly Gaussian with constant correlation coefficient $r$, then $n(\mathcal U_n-\mathcal U)=\mathcal V+n^{-\alpha/2}\mathcal Z_n$ and $(\mathcal Z_n,\mathcal B^H_1,\mathcal B^H_2)\to (\mathcal Z,\mathcal B^H_1,\mathcal B^H_2)$ in distribution, where
    \begin{equation}
    \begin{aligned}
        \mathcal V&=\begin{pmatrix}
            \frac{1}{2}\mathcal B^H_1(1)\bar {\mathcal B}_2^H+\frac{1}{2}\mathcal B^H_2(1)\bar {\mathcal B}_1^H-\frac{1}{2}\mathcal B^H_1(1)\mathcal B^H_2(1)\\
            \mathcal B^H_1(1)\bar {\mathcal B}_1^H-\frac{1}{2}\mathcal B^H_1(1)^2\\
            \mathcal B^H_2(1)\bar {\mathcal B}_2^H-\frac{1}{2}\mathcal B^H_2(1)^2
        \end{pmatrix},
    \end{aligned}
    \end{equation}
    \begin{equation}
        \mathcal Z=\sigma_H
        \begin{pmatrix}
            \int_0^1 [\bar {\mathcal B}_2^H-\mathcal B_2^H(t)]dW_1(t)+r\int_0^1 [\bar {\mathcal B}_1^H-\mathcal B_1^H(t)]dW_1(t)+\sqrt{1-r^2}\int_0^1 [\bar {\mathcal B}_1^H-\mathcal B_1^H(t)]dW_2(t)\\
            2\int_0^1 [\bar {\mathcal B}_1^H-\mathcal B_1^H(t)]dW_1(t)\\
            2r\int_0^1 [\bar {\mathcal B}_2^H-\mathcal B_2^H(t)]dW_1(t)+2\sqrt{1-r^2}\int_0^1 [\bar {\mathcal B}_2^H-\mathcal B_2^H(t)]dW_2(t)
        \end{pmatrix},
    \end{equation}where $\bar {\mathcal B}_i^H=\int_0^1 \mathcal B_i^H(t)dt$, $i\in\{1,2\}$, and, $W_i$, $i\in\{1,2\}$, are independent standard Brownian motions.
    Moreover, conditionally on $\mathcal B_1^H$ and $\mathcal B_2^H$, $\mathcal Z\sim\mathcal N(0,\Sigma_r)$, where
    \begin{equation}
        \Sigma_r=\sigma_H^2\begin{pmatrix}
            \mathcal X^{1,1}+\mathcal X^{2,2}+2r \mathcal X^{1,2}&2\mathcal X^{1,2}+2r \mathcal X^{1,1}& 2\mathcal X^{1,2}+2r \mathcal X^{2,2}\\
            2\mathcal X^{1,2}+2r \mathcal X^{1,1}&4\mathcal X^{1,1}&4r \mathcal X^{1,2}\\
            2\mathcal X^{1,2}+2r \mathcal X^{2,2}&4r \mathcal X^{1,2}& 4\mathcal X^{2,2}
        \end{pmatrix}.
    \end{equation}
\end{proposition}
\begin{proof}
    The result is a direct consequence of Lemma~\ref{lem:convergence_integral}, and
    the proof is essentially the same as that of Proposition~\ref{sprop:conj6}.
\end{proof}

\subsubsection{Empirical correlation}
\label{fsec:corr}
    Let $\mathcal U$, $\mathcal U_n$, $\mathcal V$, and $\mathcal Z$ be defined as in Proposition~\ref{fprop:conj6}. 
Recall that $F(a,b,c)=a/\sqrt{bc}$, $\rho=F(\mathcal U)$, and $\rho_n=F(\mathcal U_n)$.
    \begin{proposition}
    \label{fprop:conj7}
    An $n\to\infty$, $(n^{\alpha/2}[n(\rho_n-\rho)-\nabla F(\mathcal U)\cdot \mathcal V],\mathcal B^H_1,\mathcal B^H_2)\to(\nabla F(\mathcal U)\cdot \mathcal Z,\mathcal B^H_1,\mathcal B^H_2)$ in distribution, and, conditional on $(\mathcal B^H_1,\mathcal B^H_2)$, 
    \begin{equation}
        \nabla F(\mathcal U)\cdot \mathcal Z\sim\mathcal N\left(0,\sigma_H^2\cdot\frac{(\mathcal X^{1,1}\mathcal X^{2,2}-(\mathcal X^{1,2})^2)(\mathcal X^{1,1}+\mathcal X^{2,2}-2r\mathcal X^{1,2})}{(\mathcal X^{1,1}\mathcal X^{2,2})^2}\right).
    \end{equation}
\end{proposition}
\begin{proof}
    By Taylor's expansion, we obtain that
    \begin{equation}
        \begin{aligned}
            &n^{\alpha/2}[n(\rho_n-\rho)-\nabla F(\mathcal U)\cdot \mathcal V]\\
            &=n^{\alpha/2}[n(F(\mathcal U_n)-F(\mathcal U))-\nabla F(\mathcal U)\cdot \mathcal V]\\
            &=n^{\alpha/2}[\nabla F(\mathcal U)\cdot n(\mathcal U_n-\mathcal U)+[\nabla F(\mathcal U)-\nabla F(g(\mathcal U_n,\mathcal U))]\cdot n(\mathcal U_n-\mathcal U)-\nabla F(\mathcal U)\cdot \mathcal V]\\
            &=\nabla F(\mathcal U)\cdot n^{\alpha/2}[n(\mathcal U_n-\mathcal U)- \mathcal V]+n^{\alpha/2}[\nabla F(\mathcal U)-\nabla F(g(\mathcal U_n,\mathcal U))]\cdot n(\mathcal U_n-\mathcal U),
        \end{aligned}
    \end{equation}where $g(x,y)$ is a point on the line segment connecting $x$ and $y$.
    Due to Proposition~\ref{fprop:conj6}, it suffices to prove the second term on the right-hand side converges to $0$ in probability and verify that 
    \begin{equation}
    \label{eq:verifys}
    \nabla F(\mathcal U)^*\Sigma_r\nabla F(\mathcal U)=\frac{(\mathcal X^{1,1}\mathcal X^{2,2}-(\mathcal X^{1,2})^2)(\mathcal X^{1,1}+\mathcal X^{2,2}-2r\mathcal X^{1,2})}{(\mathcal X^{1,1}\mathcal X^{2,2})^2}.
    \end{equation}
    The verification is a simple matrix multiplication calculation, and the details are omitted here.
    
    Fix $\kappa>0$ and $\varepsilon>0$ arbitrarily small. It follows from the previous calculations that $n(\mathcal U_n-\mathcal U)$ has uniformly bounded $L^2$ norm. Hence, we can find $M>0$ sufficiently large such that for all $n$
    \begin{equation}
    \label{feq:e1}
        \Prob(n|\mathcal U_n-\mathcal U|>M)<\varepsilon/3.
    \end{equation}
    Note that $\nabla^2 F$ is continuous outside the set $S:=\{(x,y,z):xyz=0\}$, thus $\nabla F$ is Lipschitz continuous in any compact set that does not intersect with $S$. Let 
    \[L_\delta=\left\{\begin{matrix}
        x\in\mathbb R^3: |x_i|>\delta, i=1,2,3;\\
        \text{for each }y\in\mathbb R^3\text{ s.t. } |x_i-y_i|\leq|x_i|/2, i=1,2,3, |\nabla F(x)-\nabla F(y)|\leq|x-y|/\delta
    \end{matrix}\right\}.\]
    It is straightforward to see that $\Prob(\mathcal U\in S)=0$, and thus $\displaystyle\lim_{\delta\downarrow0}\Prob(\mathcal U\in L_\delta)=\Prob(\mathcal U\in \lim_{\delta\downarrow0}L_\delta)=\Prob(\mathcal U\in\mathbb R^3\setminus S)=1$, so there exists $\delta>0$ such that \begin{equation}
    \label{feq:e2}
    \Prob(\mathcal U\in L_\delta)>1-\varepsilon/3.
    \end{equation} 
    Again, since $n(\mathcal U_n-\mathcal U)$ has uniformly bounded $L^2$ norm and $0<\alpha<1$, there exists $N>0$ such that $2\kappa/M<N^{\alpha/2}$ (to be used later) and, for all $n>N$,
    \begin{equation}
    \label{feq:e3}
        \Prob(n^\alpha|\mathcal U_n-\mathcal U|>\delta\kappa/M)<\varepsilon/3.
    \end{equation}
    Combining \eqref{feq:e1}, \eqref{feq:e2}, and \eqref{feq:e3}, with $M,\delta,N$ chosen as above, we have for $n>N$
    \begin{align}
        &\Prob(n^{\alpha/2}|[\nabla F(\mathcal U)-\nabla F(g(\mathcal U_n,\mathcal U))]\cdot n(\mathcal U_n-\mathcal U)|\leq\kappa)\\
        &\quad\geq \Prob(n|\mathcal U_n-\mathcal U|\leq M, n^{\alpha/2}|\nabla F(\mathcal U)-\nabla F(g(\mathcal U_n,\mathcal U))|\leq\kappa/M)\\
        &\quad\geq \Prob(n^{\alpha/2}|\nabla F(\mathcal U)-\nabla F(g(\mathcal U_n,\mathcal U))|\leq\kappa/M)-\varepsilon/3\\
        &\quad\geq \Prob(\mathcal U\in L_\delta,|(\mathcal U)_i-(\mathcal U_n)_i|<|\mathcal U|_i/2, i=1,2,3, |\mathcal U-\mathcal U_n|\leq n^{-\alpha/2}\kappa\delta/M)-\varepsilon/3\\
        &\quad\geq \Prob(\mathcal U\in L_\delta,|\mathcal U-\mathcal U_n|\leq n^{-\alpha/2}\kappa\delta/M)-\varepsilon/3\\
        &\quad>1-\varepsilon/3-\varepsilon/3-\varepsilon/3\\
        &\quad =1-\varepsilon,
    \end{align}where the fourth inequality is because, if $\mathcal U\in L_\delta$ and $|\mathcal U-\mathcal U_n|\leq n^{-\alpha/2}k\delta/M$, then for each $i\in\{1,2,3\}$
    \begin{align}
        |(\mathcal U)_i-(\mathcal U_n)_i|\leq |\mathcal U-\mathcal U_n|\leq n^{-\alpha/2}k\delta/M<N^{-\alpha/2}k\delta/M<\delta/2<|(\mathcal U)_i|/2.
    \end{align}
    The result now follows from Proposition~\ref{fprop:conj6}.
\end{proof}

\begin{corollary}
\label{fcor:unconditional}
    Let $\mathcal U$ and $\mathcal V$ be defined as in Proposition~\ref{fprop:conj6}.
    Recalling the definition in \eqref{def:musigma}, we have \begin{gather}
    {\sigma^{r,H}(\mathcal B^H_1,\mathcal B^H_2)}^2 = \sigma_H^2\cdot\frac{(\mathcal X^{1,1}\mathcal X^{2,2}-(\mathcal X^{1,2})^2)(\mathcal X^{1,1}+\mathcal X^{2,2}-2r\mathcal X^{1,2})}{(\mathcal X^{1,1}\mathcal X^{2,2})^2}.
    \end{gather}
    Then \[
    \frac{n^{\alpha/2}(n(\rho_n-\rho)-\nabla F(\mathcal U)\cdot \mathcal V)}{\sigma^{r,H}(\mathcal B^H_1,\mathcal B^H_2)}\to\mathcal N(0,1)
    \] in distribution.
\end{corollary}
The proof of this corollary is similar to that of Corollary~\ref{scor:unconditional} and is therefore omitted.
Theorem~\ref{fthm:cov} follows immediately from Propositions~\ref{fprop:conj6} and \ref{fprop:conj7} and Corollary~\ref{fcor:unconditional}.
\section{Numerics}
\label{sec:numerics}
In this section, we present results from numerical simulations to illustrate this paper's key results.
We present both the continuous version and the discrete version of results in the standard Brownian motion case and in the fractional Brownian motion case, with difference choices of true correlation coefficients.
\begin{enumerate}
\item Figure~\ref{fig:1} (Theorem~\ref{sthm:cov}) : Standard Brownian motions with true correlation coefficient $r=0.7$.
With
\begin{equation}
    z_n = \frac{n(\rho_n-\rho)-\mu(\mathcal W_1,\mathcal W_2)}{\sigma^r(\mathcal W_1,\mathcal W_2)},
\end{equation}where the quantities requiring continuous-time observation, $\rho$, $\mu$, and $\sigma^r$, are approximated using discrete-time observation at $100\times n$ frequency.
\begin{figure}[H]
  \centering
  \begin{subfigure}[b]{0.32\textwidth}
    \centering
    \includegraphics[width=\linewidth]{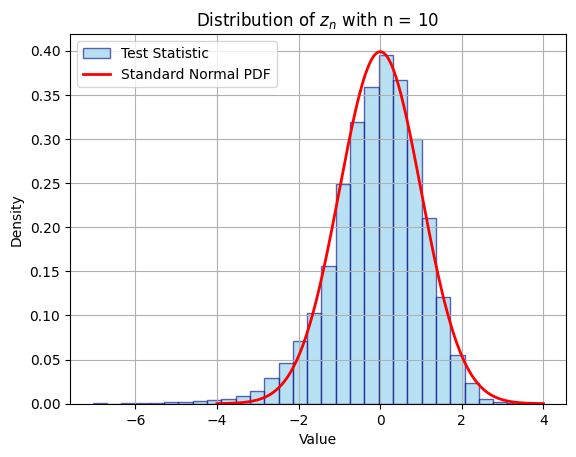}
    \caption{Mean:$-0.075$ Std:$1.081$}
  \end{subfigure}\hfill
  \begin{subfigure}[b]{0.32\textwidth}
    \centering
    \includegraphics[width=\linewidth]{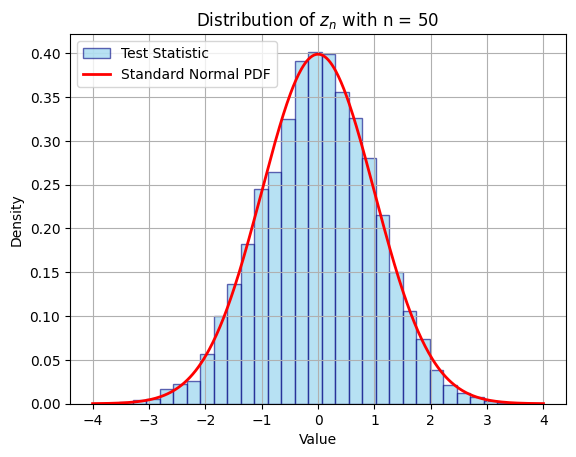}
    \caption{Mean:$-0.005$ Std:$0.999$}
  \end{subfigure}\hfill
  \begin{subfigure}[b]{0.32\textwidth}
    \centering
    \includegraphics[width=\linewidth]{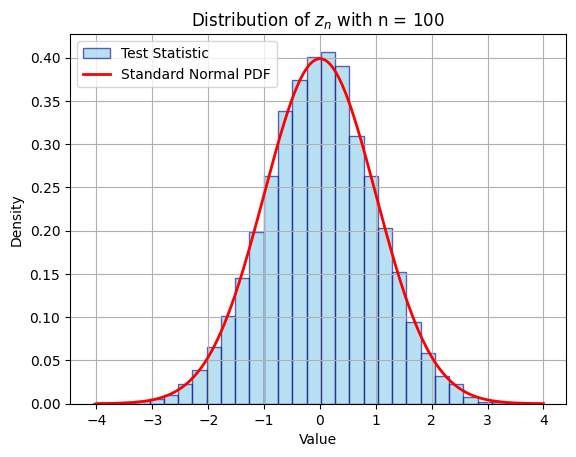}
    \caption{Mean:$-0.003$ Std:$0.982$}
  \end{subfigure}
  \caption{Standard Brownian motion: continuous version}
\label{fig:1}
\end{figure}

\item Figure~\ref{fig:2} (Proposition~\ref{sprop:discrete}) : Standard Brownian motions with true correlation coefficient $r=0.3$.
With
\begin{equation}
    z_n^d =\frac{n(\rho_n-\rho_{2n})-\mu_n(\mathcal W_1,\mathcal W_2)}{\sigma_n^r(\mathcal W_1,\mathcal W_2)}.
\end{equation}
\begin{figure}[H]
  \centering
  \begin{subfigure}[b]{0.32\textwidth}
    \centering
    \includegraphics[width=\linewidth]{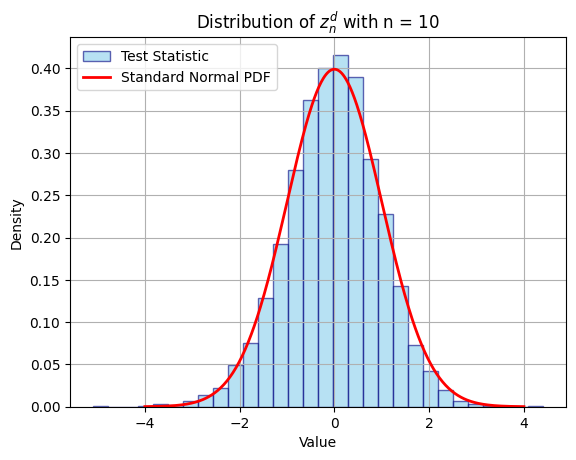}
    \caption{Mean:$-0.021$ Std:$0.979$}
  \end{subfigure}\hfill
  \begin{subfigure}[b]{0.32\textwidth}
    \centering
    \includegraphics[width=\linewidth]{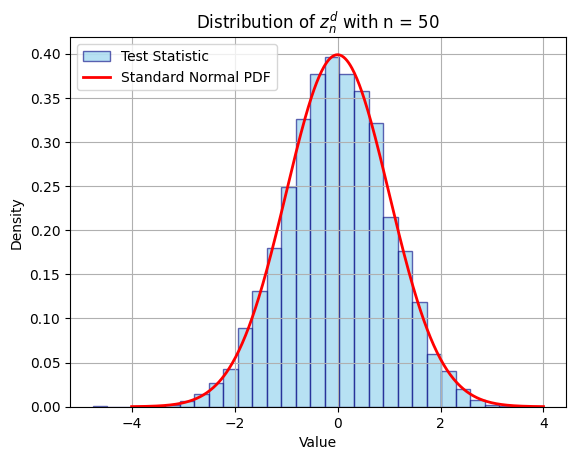}
    \caption{Mean:$-0.014$ Std:$0.995$}
  \end{subfigure}\hfill
  \begin{subfigure}[b]{0.32\textwidth}
    \centering
    \includegraphics[width=\linewidth]{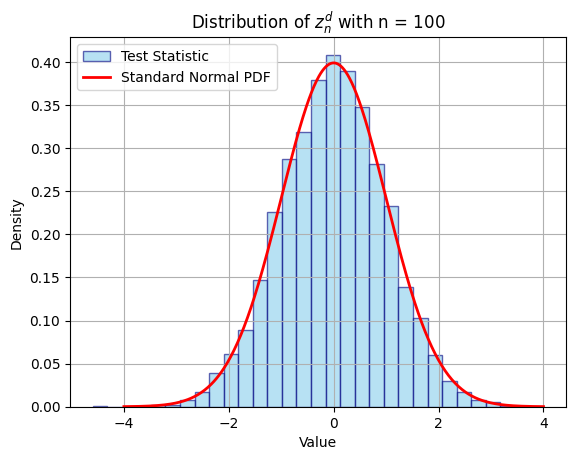}
    \caption{Mean:$-0.011$ Std:$0.991$}
  \end{subfigure}
  \caption{Standard Brownian motion: discrete version}
\label{fig:2}
\end{figure}

\item
Figure~\ref{fig:3} (Theorem~\ref{fthm:cov}) : Fractional Brownian motions with Hurst parameter $H=0.75$ and true correlation coefficient $r=0$.
With
\begin{equation}
    z_n =  \frac{n^{H+\frac{1}{2}}(\rho_n-\rho)-n^{H-\frac{1}{2}}\mu(\mathcal B^H_1,\mathcal B^H_2)}{\sigma^{r,H}(\mathcal B^H_1,\mathcal B^H_2)},
\end{equation}where the quantities requiring continuous-time observation, $\rho$, $\mu$, and $\sigma^{r,H}$, are approximated using discrete-time observation at $100\times n$ frequency.
The value of $\sigma_H$ is calculated using the formula in \eqref{def:sigmaH}, which is approximately $0.025485$ when $H=0.75$.

\begin{figure}[H]
  \centering
  \begin{subfigure}[b]{0.32\textwidth}
    \centering
    \includegraphics[width=\linewidth]{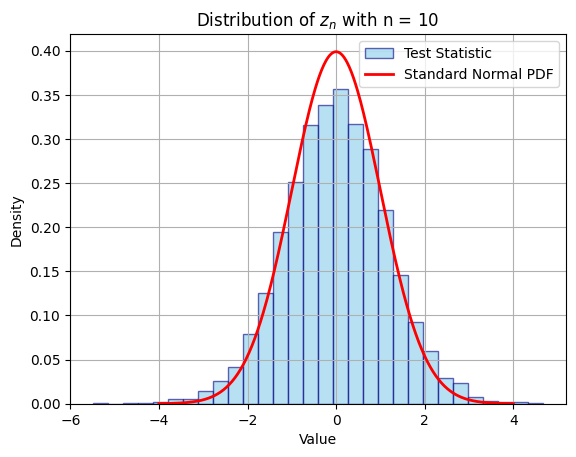}
    \caption{Mean:$-0.002$ Std:$1.145$}
  \end{subfigure}\hfill
  \begin{subfigure}[b]{0.32\textwidth}
    \centering
    \includegraphics[width=\linewidth]{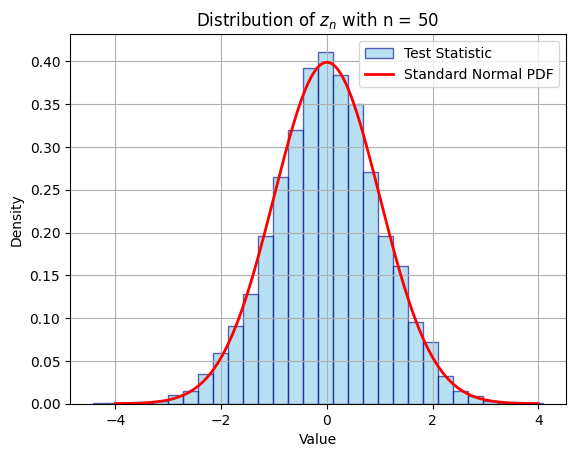}
    \caption{Mean:$0.002$ Std:$1.013$}
  \end{subfigure}\hfill
  \begin{subfigure}[b]{0.32\textwidth}
    \centering
    \includegraphics[width=\linewidth]{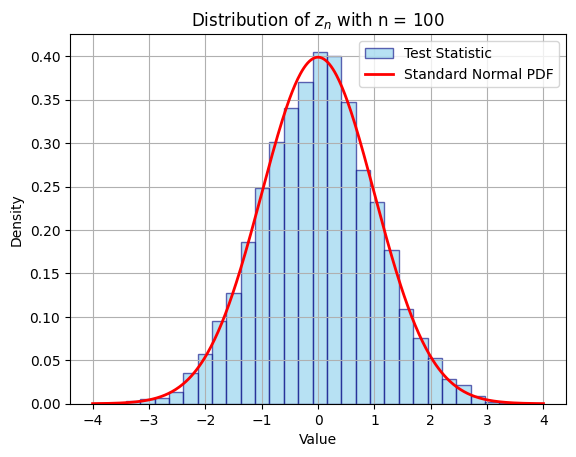}
    \caption{Mean:$-0.002$ Std:$1.010$}
  \end{subfigure}
  \caption{Fractional Brownian motion: continuous version}
  \label{fig:3} 
\end{figure}

\item
Figure~\ref{fig:4} (Proposition~\ref{fprop:discrete}) : Fractional Brownian motions with Hurst parameter $H=0.75$ and true correlation coefficient $r=-0.8$.
With
\begin{equation}
    z_n^d = \frac{n^{H+\frac{1}{2}}(\rho_n-\rho_{2n})-n^{H-\frac{1}{2}}\mu_n(\mathcal B^H_1,\mathcal B^H_2)}{\sigma_n^{r,H}(\mathcal B^H_1,\mathcal B^H_2)},
\end{equation}where the quantities requiring continuous-time observation, $\rho$, $\mu$, and $\sigma^{r,H}$, are approximated using discrete-time observation at $100\times n$ frequency.
The value of $\sigma_H^d$ is calculated using the formula in \eqref{def:sigma_Hd}, which is approximately $0.020980$ when $H=0.75$.
\begin{figure}[H]
  \centering
  \begin{subfigure}[b]{0.32\textwidth}
    \centering
    \includegraphics[width=\linewidth]{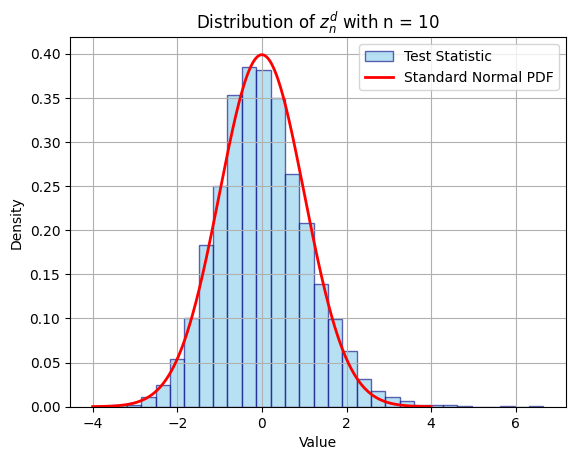}
    \caption{Mean:$0.019$ Std:$1.069$}
  \end{subfigure}\hfill
  \begin{subfigure}[b]{0.32\textwidth}
    \centering
    \includegraphics[width=\linewidth]{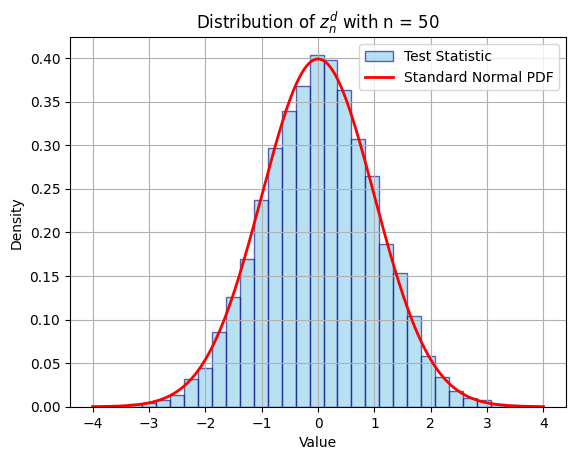}
    \caption{Mean:$0.024$ Std:$0.995$}
  \end{subfigure}\hfill
  \begin{subfigure}[b]{0.32\textwidth}
    \centering
    \includegraphics[width=\linewidth]{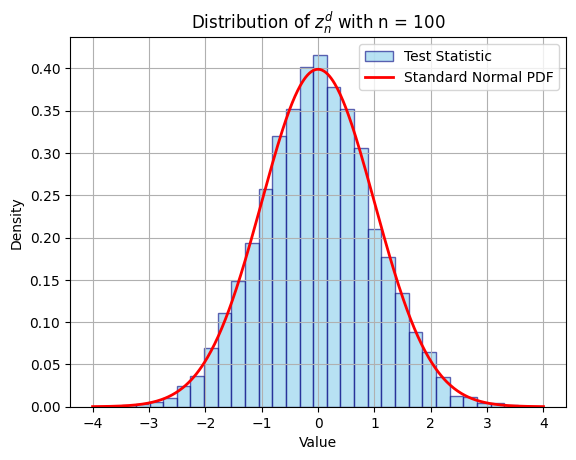}
    \caption{Mean:$-0.006$ Std:$0.994$}
  \end{subfigure}
  \caption{Fractional Brownian motion: discrete version}
\label{fig:4}
\end{figure}

\textbf{Acknowledgments}: We acknowledge with gratitude support from the U.S. Office of Naval Research (ONR) Mathematical Data Science Program (N00014-18-1-2192 and N00014-21-1-2672).
\end{enumerate}
\newpage
\printbibliography
\end{document}